\documentclass{siamart0516}

\newcommand{\cahiernumber}{04}  

\makeatletter
\def\cahierline{\gdef\@cahierline}
\cahierline{Cahier du GERAD G-\number\year-\cahiernumber}
\def\ps@firstpage{
  \def\@oddhead{\hfil\normalfont\small\@cahierline}%
  \let\@evenhead\@oddhead 
}
\def\ps@myheadings{%
  \def\@oddfoot{\hfil\normalfont\small\color{gray}{\@cahierline}}%
  \let\@evenfoot\@oddfoot 
  \def\@evenhead{\rlap{\thepage}\hfil\upshape\footnotesize\leftmark\hfil}
  \def\@oddhead{\hbox{}\hfil{\upshape\footnotesize\rightmark}\hfil\llap{\thepage}}
  \let\@mkboth\@gobbletwo
  \let\sectionmark\@gobble
  \let\subsectionmark\@gobble
}
\makeatother

\usepackage{multirow}
\usepackage[low-sup]{subdepth}
\usepackage{microtype}
\usepackage[text={5.125in,8.25in},centering]{geometry}
\usepackage{graphicx}
\newcommand{\mailto}[1]{E-mail:~\href{mailto:#1}{\texttt{#1}}}

\usepackage[textwidth=.95\marginparwidth,
  backgroundcolor=lightgray,
  linecolor=orange,
  textsize=scriptsize]{todonotes}

\numberwithin{equation}{section}
  
\makeatletter
\@mparswitchfalse  
\def\listtodoname{Todo List}
\def\listoftodos{%
  \section*{\listtodoname}\mbox{~}\par
  \@starttoc{tdo}
  }
\makeatother

\usepackage{natbib}
\bibpunct{(}{)}{;}{a}{,}{,}
\newcommand*{\doi}[1]{DOI: \href{http://dx.doi.org/\detokenize{#1}}{\detokenize{#1}}}

\makeatletter
\renewcommand\tableofcontents{%
  \section*{\contentsname}\mbox{~}\par
    \@starttoc{toc}
    }
\newcommand*\l@section[2]{%
  \ifnum \c@tocdepth >\z@
    \setlength\@tempdima{1.5em}%
    \addpenalty\@secpenalty
    \begingroup
      \parindent \z@ \rightskip \@pnumwidth
      \parfillskip -\@pnumwidth
      \leavevmode \bfseries
      \advance\leftskip\@tempdima
      \hskip -\leftskip
      #1\nobreak\hfil \nobreak\hb@xt@\@pnumwidth{\hss #2}\par
    \endgroup
  \fi}
\newcommand*\l@subsection{\@dottedtocline{2}{1.5em}{2.3em}}
\newcommand*\l@subsubsection{\@dottedtocline{3}{3.8em}{3.2em}}
\newcommand*\l@paragraph{\@dottedtocline{4}{7.0em}{4.1em}}
\newcommand*\l@subparagraph{\@dottedtocline{5}{10em}{5em}}
\makeatother

\newcommand{\R}{\mathbb{R}}
\newcommand{\minim}{\mathop{\mathrm{minimize}}}
\newcommand{\minimize}[1]{\displaystyle\minim_{#1}}
\newcommand{\maxim}{\mathop{\mathrm{maximize}}}

\newcommand{\st}{\mathop{\mathrm{subject\ to}}}

\usepackage{boxedminipage}
\newenvironment{btheorem}{%
   \setlength{\fboxsep}{4pt}%
   \par\hbox{}\noindent%
   \begin{boxedminipage}{\textwidth}
     \begin{theorem}}
    {\end{theorem}
    \end{boxedminipage}
    }

\usepackage{algorithmicx}
\usepackage{algpseudocode}
\usepackage{enumitem}
\usepackage{mathtools}
\usepackage{bbm}
\usepackage{subcaption}
\usepackage{xspace}
\usepackage{booktabs}

\crefname{subsection}{section}{sections}

\title{Implementing a smooth exact penalty function
    \\ for equality-constrained nonlinear optimization%
       \thanks{January 10, 2019; revised September 19, 2019}}
\author{%
  Ron Estrin%
  \thanks{Institute for Computational and Mathematical Engineering, Stanford  University, Stanford, CA. (\mailto{restrin@stanford.edu.})}
  \and
  Michael P. Friedlander%
    \thanks{Department of Computer Science,
    University of British Columbia, Vancouver V6T 1Z4, BC, Canada
    (\mailto{mpf@cs.ubc.ca}). Research partially supported by the Office of Naval Research [award N00014-17-1-2009].}
  \and
  Dominique Orban%
    \thanks{GERAD and Department of Mathematics and Industrial Engineering,
    \'Ecole Polytechnique, Montr\'eal, QC, Canada
    (\mailto{dominique.orban@gerad.ca}). Research supported by
    supported by NSERC Discovery Grant 299010-04.}
  \and
  Michael~A.~Saunders%
    \thanks{Systems Optimization Laboratory, Department of Management Science and Engineering, Stanford University, Stanford, CA. E-mail: \mailto{saunders@stanford.edu}. Research partially supported by the National Institute of General Medical Sciences of the National Institutes of Health [award U01GM102098].}}
\input{./macros}

\newcommand{\phis}{\phi_{\sigma}}
\newcommand{\wphis}{\widetilde{\phi}_\sigma}
\newcommand{\phisr}{\phi_{\sigma, \rho}}
\newcommand{\ys}{y_{\sigma}}
\newcommand{\wys}{\widetilde{y}_{\sigma}}
\newcommand{\Ys}{Y_{\sigma}}
\newcommand{\wYs}{\widetilde{Y}_{\sigma}}
\newcommand{\gLag}{g\subl}
\newcommand{\hLag}{H\subl}
\newcommand{\Hs}{H_{\sigma}}
\newcommand{\wHs}{\widetilde{H}_{\sigma}}
\newcommand{\gs}{g_{\sigma}}
\newcommand{\wgs}{\widetilde{g}_\sigma}
\newcommand{\ws}{w_{\sigma}}

\newcommand{\Ss}{S_{\sigma}}
\newcommand{\wSs}{\widetilde{S}_{\sigma}}

\newcommand{\augmat}{\bmat{I & A \\ A^T}}

\newcommand{\epsd}{\epsilon_d}

\newcommand*{\LSQR}{\hbox{\small LSQR}\xspace}

\newcommand*{\LSLQ}{\hbox{\small LSLQ}\xspace}
\newcommand*{\LNLQ}{\hbox{\small LNLQ}\xspace}
\newcommand*{\CG}{\hbox{\small CG}\xspace}
\newcommand*{\MINRES}{\hbox{\small MINRES}\xspace}
\newcommand*{\SYMMLQ}{\hbox{\small SYMMLQ}\xspace}
\newcommand*{\CRAIG}{\hbox{\small CRAIG}\xspace}
\newcommand*{\GMRES}{\hbox{\small GMRES}\xspace}

\def\rhalf{{\textstyle{\tfrac{1}{2} \rho}}}
\def\ind{\mathbbm{1}}

\newcommand{\Bprime}{B^{\prime}}

\begin{document}

\maketitle

\begin{center} \it Dedicated to Roger Fletcher \end{center}

\thispagestyle{firstpage}
\pagestyle{myheadings}

\begin{abstract}
  We develop a general equality-constrained nonlinear optimization algorithm based on a
  smooth penalty function proposed by \cite{Fletcher:1970}. Although it was
  historically considered to be computationally prohibitive in practice, we demonstrate
  that the computational kernels required are no more expensive than other
  widely accepted methods for nonlinear optimization. The main 
  kernel required to evaluate the penalty function and its derivatives is
  solving a structured linear system. We show how to solve this system
  efficiently by storing a single factorization each iteration when the matrices
  are available explicitly. We further show how to adapt the penalty
  function to the class of factorization-free algorithms by solving the linear system
  iteratively. The penalty function therefore has promise
  when the linear system can be solved efficiently, e.g., for PDE-constrained optimization problems where efficient preconditioners exist. 
  We discuss extensions including handling simple constraints explicitly, 
  regularizing the penalty function, and inexact evaluation of
  the penalty function and its gradients. We demonstrate the merits
  of the approach and its various features on some nonlinear programs from a
  standard test set, and some PDE-constrained optimization problems.
\end{abstract}

\section{Introduction}
\label{sec:introduction}

We consider a penalty-function approach for solving general
equality-constrained nonlinear optimization problems
\begin{equation}
  \label{eq:nlp}
  \tag{NP}
  \minimize{x \in \R^n} \enspace f(x) \enspace\st\enspace c(x)=0\ :\ y,
\end{equation}
where $f: \Real^n \to \Real$ and $c: \Real^n \to \Real^m$ are smooth
functions ($m \le n$), and $y \in \Real^m$ is the vector of Lagrange multipliers.
A smooth exact penalty function $\phis$ is used to eliminate the
constraints $c(x)=0$.  The penalty function is the Lagrangian
$L(x,y) = f(x) - c(x)\T y$ evaluated at $y = y_{\sigma}(x)$ (defined in \eqref{eq:3}) treated as a function
of $x$ depending on a parameter \(\sigma > 0\). Hence, the penalty function depends only on the primal
variables $x$.  It was first proposed by \cite{Fletcher:1970} for~\eqref{eq:nlp}.
A long-held view is that Fletcher's penalty
function is not practical because it is costly to compute
\citep{Bert:1975, ConnGoulToin:2000, NocedalW:2006}. In particular,
\citet[p.436]{NocedalW:2006} warn that ``although this merit function has
some interesting theoretical properties, it has practical
limitations\ldots". Our aim is to challenge that notion and to
demonstrate that the computational kernels are no more expensive than
other widely accepted methods for nonlinear optimization, such as
sequential quadratic programming.

The penalty function is {\em exact} because local minimizers
of~\eqref{eq:nlp} are minimizers of the penalty function for all
values of $\sigma$ larger than a finite threshold $\sigma^*$.  The
main computational kernel for evaluating the penalty function and its
derivatives is the solution of a certain saddle-point system; see \eqref{eq:aug-generic}. If
the system matrix is available explicitly, we show how to factorize it
once and re-use the factors to evaluate the penalty function and its derivatives. We also adapt the penalty function for {\em
  factorization-free} optimization by solving
the linear system iteratively. This makes the penalty function
particularly applicable for certain problem classes, such as
PDE-constrained optimization problems when good preconditioners
exist; see \cref{sec:numerical-results}.



\subsection{Related work on penalty functions}
\label{sec:related-work-smooth}

Penalty functions have long been used to solve constrained problems by
transforming them into unconstrained problems that penalize violations
of feasibility. We provide a brief overview of common penalty methods
and their relation to Fletcher's penalty $\phis(x)$. More
detail is given by \citet{PilloGrippo:1984},
\citet{ConnGoulToin:2000}, and \citet{NocedalW:2006}.

The simplest example is the quadratic penalty function
\citep{Cour:1943}, which removes the nonlinear constraints by adding
$\tfrac{1}{2} \rho \norm{c(x)}^2$ to the objective (for some
$\rho>0$). There are two main drawbacks: a
sequence of optimization subproblems must be solved with increasing
$\rho$, and a feasible solution is obtained only when
$\rho \rightarrow \infty$. As $\rho$ increases, the subproblems become increasingly
difficult to solve.

An alternative to smooth non-exact penalty functions is an exact
non-smooth function such as the 1-norm penalty $\rho
\norm{c(x)}_1$ 
\citep{Piet:1969,Flet:1985b}.
However, only non-smooth optimization methods apply, which typically
exhibit slower convergence. \cite{Mara:1978} further noted that
non-smooth merit functions may reject steps and prevent quadratic
convergence.

Another distinct approach is the class of augmented Lagrangian
methods, independently introduced by \cite{Hest:1969} and
\cite{Powe:1969}. These minimize a sequence of augmented
Lagrangians,
$L_{\rho_k}(x,y_k) = L(x, y_k)~+~\tfrac{1}{2} \rho_k \norm{c(x)}^2$. When $y_k$
is optimal, $L_{\rho_k} (x,y_k)$ is exact for sufficiently large $\rho_k$, thus avoiding the stability
issues of the quadratic penalty. However, a sequence of subproblems
must be solved to drive $y_k$ to optimality.

Although these penalty functions have often been successful in
practice, in light of their drawbacks, a class of smooth exact penalty
functions has been explored \citep{PilloGrippo:1984,
  AnitZava:2014}. With smooth exact penalty functions, constrained
optimization problems such as~\eqref{eq:nlp} can be replaced by a
single smooth unconstrained optimization problem (provided the penalty
parameter is sufficiently large).  Approximate second-order methods
can be applied to obtain at least superlinear local convergence. These
methods are variations of minimizing the augmented Lagrangian, where
either the multipliers are parametrized in terms of $x$,
or they are kept independent and the gradient of
the Lagrangian is penalized. The price for smoothness (as we find for $\phis$)
is that a derivative of the penalty
function requires a higher-order derivative from the original problem
data.  That is, evaluating $\phis$ requires $\nabla f$ and $\nabla c$;
$\nabla \phis$ requires $\nabla^2 f$ and $\nabla^2 c_i$; and so
on. The third derivative terms are typically discarded during
computation, but it can be shown that superlinear convergence is
retained \citep[Theorem 2]{Fletcher:1973}.

\cite{Fletcher:1970} introduced the class of smooth exact penalty
functions from which $\phis$ originates.  Extensions and variations of
this class have been explored by several authors, whose contributions
are described by \citet[\S 14.6]{ConnGoulToin:2000}. However,
\cite{Fletcher:1970} envisioned his method being applied to small problems,
and assumed ``the matrices in the problem are
non-sparse". Further, most developments surrounding this method
focused on linesearch schemes that require computing an explicit
Hessian approximation and using it to compute a Newton direction. One
of our goals is to show how to adapt the method to large-scale
problems by taking advantage of computational advances made since
Fletcher's proposal. Improved sparse matrix factorizations and
iterative methods for solving linear systems, and modern Newton-CG
trust-region methods \citep{Toint81b,Stei:1983}, play a key role in
the efficient implementation of his penalty function. We also show how
regularization can be used to accommodate certain constraint degeneracy, and
explain how to take advantage of inexact evaluations of functions and
gradients.


\subsection{Outline}
\label{sec:outline}

We introduce the penalty function in
\cref{sec:equality-constraints} and describe its properties and
derivatives in \cref{sec:prop-penalty-funct}. In
\cref{sec:comp-penalty-funct} we discuss options for efficiently
evaluating the penalty function and its derivatives. 
We then discuss
modifications of the penalty function in
\cref{sec:explicit-constraints}--\ref{sec:regularization} to take
advantage of linear constraints and to regularize the penalty function
if the constraint Jacobian is rank-deficient.
In some applications, it
may be necessary to solve large linear systems inexactly, and we show
in \cref{sec:inexact} how the resulting imprecision can be
accommodated. Other practical matters are described in
\cref{sec:practical-considerations}. We apply the penalty function to
several optimization problems in \cref{sec:numerical-results}, and
conclude with future research directions in \cref{sec:conclusion}.



\section{The penalty function for equality constraints}
\label{sec:equality-constraints}

For \eqref{eq:nlp}, Fletcher's penalty function is
\begin{equation}
  \label{eq:phi}
  \phis(x) := f(x) - c(x)\T \ys(x),
\end{equation}
where $\ys(x)$ are Lagrange multiplier estimates defined
with other items as
\begin{subequations}
\begin{align}
   \label{eq:3}
   \ys(x) &:= \textstyle\argmin_{y}
                   \left\{\half\norm{A(x)y - g(x)}_2^2 + \sigma c(x)\T y\right\},
                   & g(x) := \nabla f(x),  
\\ \label{eq:8}
   A(x)        &:= \nabla c(x) = \bmat{g_1(x) &\dots& g_m(x)},
                   & g_i(x):= \nabla c_i(x),
\\ \label{eq:Y}
   \Ys(x) &:= \nabla \ys(x).
\end{align}
\end{subequations}
Note that $A$ and $\Ys$ are $n$-by-$m$ matrices. The form of $\ys(x)$ is reminiscent of the variable-projection algorithm of \cite{GoluPere:1973} for separable nonlinear least-squares problems.

We assume that~\eqref{eq:nlp} satisfies some variation of the following conditions:

\smallskip


\begin{enumerate}[label=(A\arabic*)]
\item $f$ and $c$ are twice continuously differentiable and either:
  \begin{enumerate}[label=(A1\alph*)]
    \item\label{assump:c2} have Lipschitz second-derivatives, or
    \item\label{assump:c3} are three-times continuously differentiable.
  \end{enumerate}
\item The linear independence constraint qualification (LICQ) is
  satisfied at:
  \begin{enumerate}[label=(A2\alph*)]
    \item\label{assump:licq-min} stationary points of~\eqref{eq:nlp}, or
    \item\label{assump:licq-all} all $n$-vectors $x$.
  \end{enumerate}
  \hspace*{-25pt}(LICQ is satisfied at a point 
  $x$ if the vectors
  $\left\{\nabla c_i(x)\right\}_{i=1}^m$ are linearly independent.)
\item The problem is feasible, i.e., there exists $x$ such that
  $c(x) = 0$.
\end{enumerate}

\smallskip

\noindent
Assumption~\ref{assump:c3} ensures that $\phis$ has two continuous derivatives and is typical for smooth exact penalty functions \citep[Proposition 4.16]{Bert:1982}. However, we use at most two derivatives of $f$ and $c$ throughout. We typically assume~\ref{assump:c3} to simplify the discussion, but this assumption can often be weakened to~\ref{assump:c2}.
We also initially assume that~\eqref{eq:nlp} satisfies~\ref{assump:licq-all} so that $\Ys(x)$ and $\ys(x)$ are uniquely defined. We relax this assumption to~\ref{assump:licq-min} in \cref{sec:regularization}.

\subsection{Notation} 
\label{sec:notation}

Denote $\xstar$ as a local minimizer of \eqref{eq:nlp}, with corresponding dual solution $\ystar$. Let $H(x)=\nabla^2 f(x)$, $H_i(x)=\nabla^2 c_i(x)$, and define
\begin{subequations} \label{eq:grad-hess-sig}
\begin{align}
    \gLag(x,y) &:= g(x) - A(x) y,
    \qquad&
    \gs(x) &:= \gLag(x,\ys(x)),
\\  \hLag(x,y) &:= H(x) - \sum_{i=1}^m y_i H_i(x),
    \qquad&
    \Hs(x) &:= \hLag(x,\ys(x))
\end{align}
\end{subequations}
as the gradient and Hessian of $L(x,y)$ evaluated at $x$ and $y$, or
evaluated at $\ys(x)$.  We also define the matrix operators
\begingroup \allowdisplaybreaks
\begin{align*}
   S(x,v) &:= \nabla_x[A(x)\T v]
            = \nabla_x \bmat{g_1(x)^T v \\ \vdots \\ g_m(x)^T v}
            =          \bmat{v^T H_1(x) \\ \vdots \\ v^T H_m(x)},
\\
T(x,w) &:= \nabla_x[A(x) w]
            = \nabla_x \left[ \sum_{i=1}^m w_i g_i(x) \right]
            =                 \sum_{i=1}^m w_i H_i(x),
\end{align*}
\endgroup where $v \in \Real^n$, $w \in \Real^m$, and $T(x,w)$ is a
symmetric matrix.  The operation of multiplying the adjoint of $S$ with a vector $w$ is described by
\begin{align*}
S(x,v)^T w &= \left[\sum_{i=1}^m w_i H_i(x)\right] v = T(x,w) v = T(x,w)^T v \,.
\end{align*}
If $A(x)$ has full rank $m$, the operators
\begin{equation*}
  P(x) := A(x)\big(A(x)\T A(x)\big)\inv A(x)^T
  \text{and}
  \Pbar(x) := I - P(x)
\end{equation*}
define, respectively, orthogonal projectors onto $\range(A(x))$ and
its complement. More generally, for a matrix $M$, respectively define
$P_M$ and $\Pbar_M$ as the orthogonal projectors onto $\range(M)$ and
$\ker(M)$.  We define $M^{\dagger}$ as the Moore-Penrose
pseudoinverse, where $M^{\dagger} = (M\T M)^{-1} M^T$ if $M$ has full
column-rank.

Let $\lambda_{\min}(M)$ denote smallest eigenvalue of a square matrix $M$, and let $\sigma_{\min}(M)$ denote the smallest singular value for a general matrix $M$. Unless otherwise indicated, $\norm{\cdot}$ is the 2-norm for vectors and matrices. For $M$ positive definite, $\norm{u}^2_M = u\T M u$ is the energy norm. Define $\ind$ as the vector of all ones.

\section{Properties of the penalty function}
\label{sec:prop-penalty-funct}

We show how the penalty function $\phis(x)$ naturally expresses
the optimality conditions of~\eqref{eq:nlp}. We also give explicit
expressions for the threshold value of the penalty parameter $\sigma$.

\subsection{Derivatives of the penalty function}
\label{sec:derivatives-of-phi}

The gradient and Hessian of $\phis$ may be written as
\begin{subequations} \label{eq:phi-grad-hess}
  \begin{align}
    \nabla  \phis(x)  & = g_\sigma(x) - \Ys(x) c(x),
    \label{eq:phi-grad}
    \\
    \nabla^2 \phis(x) & = H_\sigma(x)- A(x) \Ys(x)^T
                                    - \Ys(x) A(x)^T
                                    - \nabla_x \left[\Ys(x) c \right],
    \label{eq:phi-hess}
  \end{align}
\end{subequations}
where 
the last term $\nabla_x[\Ys(x) c]$ 
purposely drops the argument on $c$ to emphasize that this gradient is
made on the product $\Ys(x) c$ with $c:=c(x)$ held fixed.  This
term involves third derivatives of $f$ and $c$, and as we shall see,
it is convenient and computationally efficient to ignore it.  We leave
it unexpanded.

\subsection{Optimality conditions}
\label{sec:optim-cond}

The penalty function $\phis$ is closely related to the
Lagrangian $L(x,y)$ associated with~\eqref{eq:nlp}.
To make this connection clear, we define the Karush-Kuhn-Tucker (KKT)
optimality conditions for~\eqref{eq:nlp} in terms of formulas
related to $\phis$ and its derivatives.
%
From the definition of $\phis$ and $\ys$ and the derivatives~\eqref{eq:phi-grad-hess},
the following definitions are equivalent
to the KKT conditions.

\begin{definition}[First-order KKT point] 
\label{def:kkt-1}
  A point $\xstar$ is a
  first-order KKT point of~\eqref{eq:nlp} if for any $\sigma \geq 0$ the
  following hold:
  \begin{subequations} \label{eq:1st-order}
  \begin{align}
    \label{eq:6}
    c(\xstar)                 &= 0,
  \\\nabla\phis(\xstar) &= 0.
  \end{align}
  \end{subequations}
  The Lagrange multipliers associated with $\xstar$ are $\ystar:=\ys(\xstar)$.
\end{definition}


\begin{definition}[Second-order KKT point]
\label{def:kkt-2}
  The first-order KKT point
  $\xstar$ satisfies the second-order necessary KKT condition
  for~\eqref{eq:nlp} if for any $\sigma \geq 0$,
  \begin{equation*} \label{eq:2nd-order-nec}
    p^T \nabla^2\phis(\xstar) p \ge 0
    \quad
    \hbox{for all $p$ such that }
    A(\xstar)\T p = 0,
  \end{equation*}
  i.e., \(\bar{P}(\xstar) \nabla^2\phis(\xstar) \bar{P}(\xstar) \succeq 0\).
  The condition is sufficient if the inequality is strict.
\end{definition}

The second-order KKT condition says that at $\xstar$, $\phis$ has nonnegative curvature along directions in the tangent space of the constraints. However, at $\xstar$, increasing $\sigma$ will increase curvature along the normal cone of the feasible set. We derive a threshold value for $\sigma$ that causes $\phis$ to have nonnegative curvature at a second-order KKT point $\xstar$, as well as a condition on $\sigma$ that ensures stationary points of $\phis$ are primal feasible. For a given first- or second-order KKT pair $(\xstar,\ystar)$ of \eqref{eq:nlp}, we define
\begin{equation}
  \label{eq:14}
  \sigma^* := \half\lambda^+_{\max}\left( P(\xstar) \hLag(\xstar,\ystar)P(\xstar) \right),
\end{equation}
where $\lambda^+_{\max}(\cdot) = \max\left\{\lambda_{\max}(\cdot),0\right\}$. 

\begin{lemma}
If $c(x) \in \range(A(x)^T)$, then $\ys(x)$ satisfies
\begin{equation}
    \label{eq:lin-sys-ys}
    A(x)\T A(x) \ys(x) = A(x)\T g(x) - \sigma c(x).
\end{equation}
Further, if $A(x)$ has full rank, then
\begin{equation}
    \label{eq:AAYTx}
    A(x)\T A(x) \Ys(x)\T = A(x)\T \left[ \Hs(x) - \sigma I \right] + S(x,\gs(x)).
\end{equation}
\end{lemma}
\begin{proof}
For any $x$, the necessary and sufficient optimality conditions for~\eqref{eq:3} give~\eqref{eq:lin-sys-ys}. By differentiating both sides of~\eqref{eq:lin-sys-ys}, we obtain
\[
   S(x,A(x) \ys(x)) + A(x)\T \left[ T(x,\ys(x)) + A(x) \Ys(x)^T\right]
 = S(x,g(x)) + A(x)^T[H(x) - \sigma I].
\]
From definitions \eqref{eq:grad-hess-sig}, we obtain \eqref{eq:AAYTx}.
\end{proof}

\begin{btheorem}[Threshold penalty value]
  \label{thm:threshold}

   Suppose $\nabla\phis(\xbar)=0$ for some $\xbar$,
   and let $\xstar_1$ and $\xstar_2$ be a first-order
   and a second-order necessary KKT point, respectively, for \eqref{eq:nlp}.
   Let \(\sigma^*\) be defined as in~\eqref{eq:14}.
   Then
  \begin{subequations}
  \begin{align}
    \label{eq:17}
    \sigma > \norm{A(\xbar)^T \Ys(\xbar)}
       &\quad\Longrightarrow\quad g(\xbar) = A(\xbar) y_{\sigma}(\xbar),
                            \quad c(\xbar) = 0;
\\  \label{eq:13}
    \sigma \ge \norm{A(\xstar_1) \Ys(\xstar_1)^T}
       &\quad\Longrightarrow\quad \sigma \ge \sigma^*;
\\  \label{eq:16}
    \nabla^2\phis(\xstar_2) \succeq 0
       &\quad\!\Longleftrightarrow\quad \sigma \ge \sigmabar := \half\lambda_{\max}\left( P(\xstar) \hLag(\xstar,\ystar)P(\xstar) \right).
  \end{align}
  \end{subequations}
  If $\xstar_2$ is second-order sufficient, then the inequalities
  in~\eqref{eq:16} hold strictly. Observe that $\sigma^* = \max\{\sigmabar,0\}$ and that $\sigmabar$ could be negative.
\end{btheorem}
\begin{proof}
  We prove \eqref{eq:17}, \eqref{eq:16}, and \eqref{eq:13} in
  order.

Proof of~\eqref{eq:17}: The condition $\nabla\phis(\xbar)=0$
implies that $\ys(\xbar)$ is well defined and $c(\xbar)\in\mbox{range}(A(\xbar)^T)$, so that
\[
  g(\xbar) = A(\xbar) \ys(\xbar) + \Ys(\xbar) c(\xbar).
\]
Substituting~\eqref{eq:lin-sys-ys} evaluated at $\xbar$ into this equation yields, after
simplifying,
\[
  A(\xbar)\T \Ys(\xbar) c(\xbar) = \sigma c(\xbar).
\]
Taking norms of both sides and using submultiplicativity gives the
inequality $\sigma\norm{c(\xbar)}\le\norm{A(\xbar)^T
  \Ys(\xbar)}\,\norm{c(\xbar)}$, which immediately implies that
$c(\xbar) = 0$. The condition $\nabla \phi_{\sigma}(\xbar) = 0$ then
becomes $g_{\sigma}(\xbar) = 0$.

Proof of~\eqref{eq:16}: Because $\xstar_2$ satisfies~\eqref{eq:1st-order}, we have $g_\sigma(\xstar_2)=0$ and $y^* = \ys(\xstar_2)$, independently of $\sigma$. It follows from \eqref{eq:AAYTx}, $\hLag(\xstar_2,\ystar) = \Hs(\xstar_2)$, $S(\xstar_2, \gs(\xstar_2)) = 0$, and the definition of the projector $P = P(\xstar_2)$ that
\begin{equation}
  \label{eq:11}
  \begin{split}
  A(\xstar_2) \Ys(\xstar_2)^T &= A(\xstar_2) \left(A(\xstar_2)\T A(\xstar_2)\right)^{-1}\! A(\xstar_2)\T [\Hs (\xstar_2) - \sigma I] \\ &= P (\hLag(\xstar_2, \ystar)-\sigma I).
  \end{split}
\end{equation}
We substitute this equation into~\eqref{eq:phi-hess} and use the
relation $P + \bar{P} = I$ to obtain
\begin{equation*}
  \begin{split}
  \label{eq:12}
  \nabla^2\phis(\xstar_2) &= \hLag(\xstar_2, \ystar) - P\hLag(\xstar_2,\ystar) - \hLag(\xstar_2,\ystar)P + 2\sigma P \\&= \Pbar \hLag(\xstar_2,\ystar) \Pbar
                                -  P    \hLag(\xstar_2,\ystar) P + 2\sigma P.
  \end{split}
\end{equation*}
Note that $\Pbar \hLag(\xstar_2,\ystar) \Pbar \succeq 0$ because $\xstar_2$ is a second-order KKT point, so $\sigma$ needs to be sufficiently large that $2 \sigma P - P \hLag(\xstar_2,\ystar) P \succeq 0$, which is equivalent to $\sigma \geq \sigmabar$.

Proof of~\eqref{eq:13}:  With $\xstar_1$ in \eqref{eq:11}, $y^* = \ys(\xstar_1)$, and the properties of $P$, we have
\begin{align*}
  \sigma \ge \norm{A(\xstar_1) \Ys(\xstar_1)^T}
         &=   \norm{P(\hLag(\xstar_1,\ystar)-\sigma I)}
 \\      &\ge \norm{P(\hLag(\xstar_1,\ystar)-\sigma I)P}
 \\      &\ge \norm{P \hLag(\xstar_1,\ystar) P}-\sigma \norm{P}
          \ge 2 \sigma^* - \sigma.
\end{align*}
Thus, $\sigma \ge \sigma^*$ as required.
\end{proof}

According to \cref{eq:16}, if $\xstar$ is a second-order KKT point,
there exists a threshold value $\sigmabar$ such that $\phis$ has
nonnegative curvature for $\sigma \geq \sigmabar$.
As penalty parameters are typically nonnegative, we treat $\sigma^* = \max\{\sigmabar,0\}$ as the threshold.
Unfortunately, as for many exact penalty functions, \cref{thm:threshold} allows the possibility of stationary points of
$\phis(x)$ that are not feasible points of \eqref{eq:nlp};
for an
example, see \cref{sec:spurious-min}. However, we rarely encounter
this in practice with feasible problems, and minimizers of $\phis(x)$ usually correspond to
feasible (and therefore optimal) points of \eqref{eq:nlp}.

\subsection{Additional quadratic penalty}
\label{sec:quadratic-penalty}
In light of \cref{thm:threshold}, it is somewhat unsatisfying that local minimizers of $\phis(x)$ might not be local minimizers of \eqref{eq:nlp}. 
We may add a quadratic penalty term to promote feasibility, and under mild conditions ensure that minimizers of $\phis$ are KKT points of \eqref{eq:nlp}. Like \cite{Fletcher:1970}, we define
\begin{equation}
\label{eq:quadratic-penalty}
    \phisr(x) := \phis(x) + \rhalf \norm{c(x)}^2
    = f(x) - [\ys(x) - \rhalf c(x)]\T c(x). 
\end{equation}
The multiplier estimates are now shifted by the constraint violation, similar to an augmented Lagrangian. All expressions for the derivatives follow as before with an additional term from the quadratic penalty. 

\begin{btheorem}[Threshold penalty value for quadratic penalty]
  \label{thm:threshold-quadratic}

Let $\Sscr \subset \R^n$ be a compact set, and suppose that
$\sigma_{\min}(A(x)) \geq \lambda > 0$ for all $x \in \Sscr$. Then for
any $\sigma \geq 0$ there exists $\rhostar(\sigma) > 0$ such
that for all $\rho > \rhostar(\sigma)$, if $\nabla \phisr(\xbar) = 0$
and $\xbar \in \Sscr$, then $\xbar$ is a first-order KKT point for
\eqref{eq:nlp}.
\end{btheorem}
\begin{proof}
The condition $\nabla \phisr(\xbar) = 0$ implies that
\begin{equation*}
    g(\xbar) - A(\xbar) \ys(\xbar) - \Ys(\xbar) c(\xbar) = \rho A(\xbar) c(\xbar).
\end{equation*}
We premultiply with \(A(\xbar)\T\) and use \eqref{eq:lin-sys-ys} to obtain
\begin{equation}
\label{eq:18}
\left( \sigma I - A(\xbar)^T \Ys(\xbar) \right) c(\xbar)  = \rho A(\xbar)\T A(\xbar) c(\xbar).
\end{equation}
The left-hand side of \eqref{eq:18} is a continuous matrix function
with finite supremum
${R(\sigma) := \sup_{x \in \Sscr} \norm{\sigma I - A(x)^T \Ys(x)}}$ defined
over the compact set $\Sscr$. We now define
$\rhostar(\sigma) := R(\sigma)/\lambda^2$, so that for
$\rho > \rhostar(\sigma)$,
\begin{align*}
    R(\sigma) \norm{c(\xbar)} &\ge
    \norm{\sigma I - A(\xbar)^T \Ys(\xbar) \norm{\cdot} c(\xbar)} 
\\  &\ge \norm{\left( \sigma I - A(\xbar)^T \Ys(\xbar) \right) c(\xbar)}
\\  &= \rho \norm{A(\xbar)^T A(\xbar) c(\xbar)} \ge \rho \lambda^2 \norm{c(\xbar)}.
\end{align*}
The above inequality only holds when $c(\xbar) = 0$ because $\rho\lambda^2 > R(\sigma)$, so $\xbar$ is feasible for \eqref{eq:nlp}.
Because $c(\xbar) = 0$ and $\nabla \phis(\xbar) = \nabla \phisr(\xbar) = 0$, $\xbar$ is a first-order KKT point.
\end{proof}

We briefly consider the case $\sigma = 0$ and $\rho > 0$. The threshold value to ensure positive semidefiniteness of $\nabla^2 \phisr$ at a second-order KKT pair $(\xstar, \ystar)$ to \eqref{eq:nlp} is
$$\rhostar = \lambda_{\max}^+\left(A(\xstar)^{\dagger} \hLag(\xstar, \ystar) \left(A(\xstar)^{\dagger}\right)^T \right).$$
\cite{Fletcher:1970} gives an almost identical (but slightly looser) expression for $\rhostar$.
This threshold parameter is more difficult to interpret in terms of the problem data compared to $\sigma^*$ due to the pseudoinverse.
We give a theorem analogous to \cref{thm:threshold}.

\begin{btheorem}
   \label{thm:threshold-quadratic-2}

   Suppose $\sigma = 0$ and $\rho \geq 0$. Let $\nabla\phisr(\xbar)=0$ for some $\xbar$,
   and let $\xstar$ be a second-order necessary KKT point for \eqref{eq:nlp}.
   Then
  \begin{subequations}
  \begin{align}
    \rho > \norm{A(\xbar)^{\dagger} \Ys(\xbar)}
       &\enspace\Longrightarrow\enspace g(\xbar) = A(\xbar) y_{\sigma}(\xbar),
                            \quad c(\xbar) = 0;
\\  \label{eq:16-2}
    \nabla^2\phis(\xstar) \succeq 0
       &\enspace\!\Longleftrightarrow\enspace \rho \ge \rhobar := \lambda_{\max}(A(\xstar)^{\dagger} \hLag(\xstar, \ystar) \left(A(\xstar)^{\dagger}\right)^T ).
  \end{align}
  \end{subequations}
  If $\xstar$ is second-order sufficient, the inequalities
  in~\eqref{eq:16-2} hold strictly.
\end{btheorem}
\begin{proof}
The proof is identical to that of \cref{thm:threshold}.
\end{proof}

Using $\rho > 0$ can help cases where attaining feasibility is
problematic for moderate values of $\sigma$. For simplicity
we let $\rho=0$ from now on, because it is trivial to
evaluate $\phisr$ and its derivatives if one can compute $\phis$.

\subsection{Scale invariance}
Note that $\phis$ is invariant under diagonal scaling of the constraints, i.e., if $c(x)$ is replaced by $D c(x)$ for some diagonal matrix $D$, then $\phis$ is unchanged. It is an attractive property for $\phis$ and $\sigma^*$ to be independent of some choices in model formulation, like the Lagrangian.
However, $\phisr$ with $\rho > 0$ is not scale invariant, like the augmented Lagrangian, because of the quadratic term. Therefore, constraint scaling is an important consideration if $\phisr$ is to be used.



\section{Evaluating the penalty function}
\label{sec:comp-penalty-funct}

The main challenge in evaluating $\phis$ and its gradient is solving
the shifted least-squares problem~\eqref{eq:3} in order to
compute $\ys(x)$, and computing the Jacobian $\Ys(x)$. Below we show
it is possible to compute products $\Ys(x)v$
and $\Ys(x)\T u$ by solving structured linear systems involving the
matrix used to compute \(y_{\sigma}(x)\). If direct methods are used, a single
factorization that gives the solution~\eqref{eq:3} is sufficient for all products.

For this section, it is convenient to drop the arguments on the various functions and assume they are all evaluated at a point $x$ for some parameter $\sigma$.  For example,
\(
\ys = \ys(x), \ A = A(x), \ \Ys = \Ys(x), \ H_\sigma = H_\sigma(x), \ S_\sigma = S(x,g_\sigma(x))\), etc.  We also express~\eqref{eq:AAYTx} using the shorthand notation
\begin{equation}
  \label{eq:ATAY}
  A\T A \Ys^T = A^T[H_\sigma - \sigma I] + S_\sigma.
\end{equation}
We first describe how to compute products $\Ys u$ and $\Ys\T v$, then describe how to put those pieces together to evaluate the penalty function and its derivatives.

\subsection{Computing the product $\mathbf{\Ys u}$}

For a given $u \in \Real^m$, we premultiply~\eqref{eq:ATAY} by \(u\T (A\T A)^{-1}\) to obtain
\begin{align*}
    \Ys u &= [\Hs - \sigma I] A(A\T A)^{-1}u + S_{\sigma}^T (A\T A)^{-1} u
\\     &= [\Hs - \sigma I] v - S_{\sigma}^T w,
\end{align*}
where we define $w = -(A\T A)^{-1}u$ and $v = -Aw$. 
Observe that $w$ and $v$ solve the system
\begin{equation}
  \label{eq:aug-Yu}
  \augmat \bmat{v\\w} = \bmat{
0\\u}.
\end{equation}
\autoref{alg:Yu} formalizes the process.

\begin{algorithm}[hbt]
  \caption{%
    Computing the matrix-vector product $\Ys u$
    \label{alg:Yu}
  }
  \begin{algorithmic}[1]
    \State $(v,w) \gets \hbox{solution of~\eqref{eq:aug-Yu}}$
    \State \Return $[H_\sigma - \sigma I] v - S_\sigma\T w$\;
  \end{algorithmic}
\end{algorithm}

\subsection{Computing the product $\mathbf{\Ys\T v}$}
Multiplying both sides of~\eqref{eq:ATAY} on the right by
$v$ gives
\[
  A\T A (\Ys\T v) = A^T([H_\sigma-\sigma I]v) + (S_\sigma v).
\]
The required product $u=\Ys\T v$ is in the solution of the system
\begin{equation}
  \label{eq:aug-YTv}
  \bmat{I & A \\ A^T & } \bmat{r \\ u}
  =
  \bmat{[H_\sigma-\sigma I]v \\ -S_\sigma v}.
\end{equation}
\cref{alg:YTv} formalizes the process.

\begin{algorithm}[hbt]
  \caption{%
    Computing the matrix-vector product $\Ys\T v$
    \label{alg:YTv}
  }
  \begin{algorithmic}[1]
    \State Evaluate $[H_\sigma - \sigma I] v$ and $S_\sigma v$
    \State $(r,u) \gets \hbox{solution of~\eqref{eq:aug-YTv}}$
    \State \Return $u$
  \end{algorithmic}
\end{algorithm}

\subsection{Computing multipliers and first derivatives}
The multiplier estimates $\ys$ can be obtained from the optimality conditions for~\eqref{eq:3}:
\begin{equation}
  \label{eq:aug-mult}
  \begin{bmatrix} I & A \\ A^T \end{bmatrix}
  \begin{bmatrix} \gs \\ \ys \end{bmatrix}
  =
  \begin{bmatrix} g \\ \sigma c\end{bmatrix},
\end{equation}
which also gives $\gs$. \cref{alg:Yu} then gives $\Ys c$ and hence $\nabla \phis$ in \cref{eq:phi-grad}.

Observe that we can re-order operations to take advantage of specialized solvers. Consider the pair of systems
\begin{equation}
  \label{eq:aug-pair}
  \begin{bmatrix} I & A \\ A^T \end{bmatrix}
  \begin{bmatrix} d \\ y \end{bmatrix}
  =
  \begin{bmatrix} g \\ 0 \end{bmatrix}
  \qquad \mbox{and} \qquad
  \begin{bmatrix} I & A \\ A^T \end{bmatrix}
  \begin{bmatrix} v \\ w \end{bmatrix}
  =
  \begin{bmatrix} 0 \\ c \end{bmatrix}.
\end{equation}
We have $\gs = d + \sigma v$ and $\ys = y + \sigma w$, while the computation of $\Ys c$ is unchanged. The systems in \eqref{eq:aug-pair} correspond to pure least-squares and least-norm problems respectively. Specially tailored solvers may be used to improve efficiency or accuracy. This is further explored in \cref{sec:solve-aug}. 

\subsection{Computing second derivatives}
\label{sec:second-derivatives}
We can approximate $\nabla^2 \phis$ using~\eqref{eq:phi-hess} and~\eqref{eq:AAYTx} in two ways according to
\begin{subequations}
\label{eq:hess-approx}
\begin{align}
    \nabla^2\phis &\approx B_1 := \Hs - A\Ys^T - \Ys A^T \label{eq:hess-approx-1}
\\  &\phantom{\approx B_1 :}= \Hs - \widetilde{P} \Hs - \Hs \widetilde{P} + 2\sigma \widetilde{P} - A(A^T A)^{-1} S_{\sigma} -  S_{\sigma}^T (A^T A)^{-1}A \nonumber
\\  \nabla^2\phis &\approx B_2 := \Hs - \widetilde{P} \Hs - \Hs \widetilde{P} + 2 \sigma \widetilde{P}, \label{eq:hess-approx-2}
\end{align}
\end{subequations}
where $\widetilde{P} = A(A\T A)^{-1} A\T$. Note that
$\widetilde{P} = P_A$ here, but this changes when regularization is
used; see \cref{sec:regularization}.  The first approximation ignores
$\nabla[\Ys(x)c]$ in~\eqref{eq:phi-hess}, while the second ignores
$S_{\sigma} = S(x,\gs(x))$. Because we expect $c(x)\to0$ and
$\gs(x) \rightarrow 0$, $B_1$ and $B_2$ are similar to Gauss-Newton
approximations to $\nabla^2\phis(x)$, and as \citet[Theorem 2]{Fletcher:1973}
shows, using them in a Newton-like scheme is sufficient for quadratic
convergence if \ref{assump:c2} is satisfied.

Because $\widetilde{P}$ is a projection on $\range(A)$, we can compute products $\widetilde{P} u$ by solving
\begin{equation}
  \label{eq:aug-proj}
  \begin{bmatrix} I & A \\ A^T \end{bmatrix}
  \begin{bmatrix} p \\ q \end{bmatrix}
  =
  \begin{bmatrix} u \\ 0 \end{bmatrix}
\end{equation}
and setting $\widetilde{P} u \gets u - p$. Note that with regularization, the $(2,2)$ block of this system is modified and $\widetilde{P}$ is no longer a projection; see \cref{sec:regularization}.

The approximations~\eqref{eq:hess-approx-1} and~\eqref{eq:hess-approx-2} trade Hessian accuracy for computational efficiency. If the operator $S(x,v)$ is not immediately available (or not efficiently implemented), it may be avoided. Using $B_2$ requires only least-square solves, which allows us to apply specialized solvers (e.g., \LSQR \citep{PaigSaun:1982}), which cannot be done when products with $\Ys^T$ are required.

\subsection{Solving the augmented linear system}
\label{sec:solve-aug}

We discuss some approaches to solving linear systems of the form
\begin{equation}
  \label{eq:aug-generic}
  \Kscr
  \begin{bmatrix} p \\ q \end{bmatrix}
  =
  \begin{bmatrix} w \\ z \end{bmatrix}, 
  \qquad \Kscr := \begin{bmatrix} I & A \\ A^T & -\delta^2 I \end{bmatrix},
\end{equation}
which have repeatedly appeared in this section. Although $\delta = 0$ so far, we look ahead to regularized systems as they require only minor modification. Let $(\pstar, \qstar)$ solve \eqref{eq:aug-generic}. 

Conceptually it is not important how this system is solved as long as
it is with sufficient accuracy. However, 
this is the most computationally intensive part of using
$\phis$. Different solution methods have different advantages and
limitations, depending on the size and sparsity of $A$, whether $A$ is
available explicitly, and the prescribed solution accuracy.


One option is direct methods: factorize $\Kscr$ once per iteration and use the factors to solve with each right-hand side. Several factorization-based approaches can be used with various advantages and drawbacks; see the supplementary materials for details.

In line with the goal of creating a factorization-free solver for minimizing
$\phis$, we discuss iterative methods for solving
\eqref{eq:aug-generic}, particularly Krylov subspace solvers. This
approach has two potential advantages: if a good preconditioner
$\Pscr \approx A\T A$ is available, then solving
\eqref{eq:aug-generic} could be much more efficient than with direct
methods, and we can take advantage of solvers using inexact function
values, gradients or Hessian products by solving
\eqref{eq:aug-generic} approximately; see \citet{HeinVice:2001} and \citet{KourHeinRidzBloe:2014}.
For example, \citet{PearStolWath:2012} and \citet{Simo:2012} describe various preconditioners for saddle-point systems arising in PDE-constrained optimization, which are closely related to the augmented systems in \eqref{eq:aug-generic}.

When $z = 0$, \eqref{eq:aug-generic} is a (regularized) least-squares problem: $\min_q \norm{A q - w} + \delta \norm{q}^2$. 
We use \LSQR \citep{PaigSaun:1982}, which ensures that the error in iterates $p_k$ and $q_k$ decreases monotonically at every iteration.  (\cite{HestStie:1952} show this for CG, and \LSQR is equivalent to \CG on the normal equations.) Furthermore, \cite{EstOrbSaun:2019} provide a way to compute an upper bound on $\norm{\pstar - p_k}$ and $\norm{\qstar - q_k}$ via \LSLQ when given an underestimate of $\sigma_{\min}(A\Pscr^{-1/2})$. (Note that the error norm
for $q$ depends on the preconditioner.) Further discussion is in \cref{sec:numerical-results}.

When $w = 0$, \eqref{eq:aug-generic} is a least-norm problem:
$\min_{p,s} \norm{p}^2 + \norm{s}^2 \mbox{ s.t. } A^T p + \delta s = z$. We then use \CRAIG
\citep{Craig:1955} because it minimizes the error in each Krylov
subspace. Given the same underestimate of
$\sigma_{\min}(A\Pscr^{-1/2})$, \citet{Ario:2013} and
\citet{EstOrbSaun:2019LNLQ} give a way to bound the error norms for $p$
and $q$.

Recall that $\phis$ and $\nabla \phis$ can be computed by solving only
least-squares and least-norm problems (only one of $w$ and $z$ is
nonzero at a time). Furthermore, if \eqref{eq:hess-approx-2} is used,
the remaining solves with $\Kscr$ are least-squares
solves. 
If both $w$ and $z$ are nonzero (for products with $\Ys^T$), we can
shift the right-hand side of \eqref{eq:aug-generic} and solve the
system
\begin{equation*}
    \Kscr \bmat{\pbar \\ q} = \bmat{0 \\ z - A^T w}, \qquad p = \pbar + w.
\end{equation*}
Thus, \eqref{eq:aug-generic} can be solved by \CRAIG or \LNLQ \citep{Ario:2013,EstOrbSaun:2019LNLQ} or other least-norm solvers. 

Although $\Kscr$ is symmetric indefinite, we do not recommend methods such as \MINRES or \SYMMLQ \citep{PaigSaun:1975}. \cite{OrbaArio:2017} show that if full-space methods are applied directly to $\Kscr$ 
then every other iteration of the solver makes little progress.
However, if solves with $\Pscr$ can only be performed approximately, it may be necessary to apply flexible variants of nonsymmetric full-space methods to $\Kscr$, such as flexible \GMRES \citep{Saad:1993}.

\section{Maintaining explicit constraints}
\label{sec:explicit-constraints}
We consider a variation of \eqref{eq:nlp} where some of the constraints $c(x)$ are easy to maintain explicitly; for example, some are linear. 
We can then maintain feasibility for a subset of the constraints, the contours of the $\phis$ are simplified, and as we show soon, the threshold penalty parameter $\sigma^*$ is decreased. We discuss the case where some of the constraints are linear, but it is possible to extend the theory to any type of constraint.

Consider the problem
\begin{equation}
    \label{eq:nlp-exp}
    \tag{NP-EXP}
    \minimize{x \in \R^n} \enspace f(x) \enspace\st\enspace c(x)=0, \enspace B\T x = d,
\end{equation}
where we have nonlinear constraints $c(x) \in \R^{m_1}$ and linear constraints $B\T x = d$ with $B \in \R^{n \times m_2}$, so that $m_1 + m_2 = m$. We assume that \eqref{eq:nlp-exp} at least satisfies \ref{assump:licq-min}, so that $B$ has full column rank.
We define the penalty function problem to be
\begin{equation*}
\begin{aligned}
    \minimize{x \in \R^n} \quad \phis (x) &:= f(x) - c(x)^T \ys(x) \quad\st\quad B\T x = d,
\\  \bmat{\ys(x) \\ \ws(x)} &:= \argmin_{y,w} \half \norm{A(x) y + Bw - g(x)}^2 + \sigma \bmat{c(x) \\ B\T x - d}^T \bmat{y \\ w},
\end{aligned}
\end{equation*}
which is similar to \eqref{eq:phi} except the linear constraints are not penalized in $\phis(x)$, and the penalty function is minimized subject to the linear constraints. A possibility is to also penalize the linear constraints, while keeping them explicit; however, penalizing the linear constraints in $\phis(x)$ introduces additional nonlinearity, and if all constraints are linear, it makes sense that the penalty function reduces to~\eqref{eq:nlp-exp}.

For a given first- or second-order KKT solution $(\xstar,\ystar)$, the threshold penalty parameter becomes 
\begin{equation}
    \label{eq:threshold-eq}
    \sigma^* := \half \lambda^+_{\max}\left( \Pbar_B P_C \hLag(\xstar,\ystar) P_C \Pbar_B \right) \leq \half \lambda^+_{\max} \left( P_C \hLag(\xstar,\ystar) P_C \right),
 \end{equation}
where $C(x) = \bmat{A(x) & B}$ is the Jacobian for all constraints. Inequality \eqref{eq:threshold-eq} holds because $\Pbar_B$ is an orthogonal projector. If the linear constraints were not explicit, the threshold value would be the right-most term in \eqref{eq:threshold-eq}. Intuitively, the threshold penalty value decreases by the amount of the top eigenspace of the Lagrangian Hessian that lies in the range of $B^T$, because positive semidefiniteness of $\nabla^2 \phis(\xstar)$ along that space is guaranteed by the underlying solver.

It is straightforward to adapt \cref{thm:threshold} to obtain an analogous exact penalization results for the case with explicit constraints.





\section{Regularization}

\label{sec:regularization}
Even if $A(\xstar)$ has full column rank, $A(x)$ might have low column rank or small singular values away from the solution. If $A(x)$ is rank-deficient and $c(x)$ is not in the range of $A(x)^T$, then $\ys(x)$ and $\phis(x)$ are undefined. Even if $A(x)$ has full column rank but is close to rank-deficiency, the linear systems \eqref{eq:aug-Yu}--\eqref{eq:aug-mult} and \eqref{eq:aug-proj} are ill-conditioned, threatening inaccurate solutions and impeded convergence.

We modify $\phis$ by changing the definition of the multiplier estimates in \eqref{eq:3} to solve a regularized shifted least-squares problem with regularization parameter $\delta > 0$:
\begin{subequations}
\begin{align}
  \label{eq:pen-reg}
    \phis(x;\delta) &:= f(x) - c(x)\T \ys(x;\delta)
\\  \ys(x;\delta) &:= \argmin_y \ \half\norm{A(x)y - g(x)}_2^2 + \sigma c(x)\T y + \half \delta^2 \norm{y}_2^2.
\end{align}
\end{subequations}
This modification is similar to the exact penalty function of \cite{PillGrip:1986}.
The regularization term $\half \delta^2 \norm{y}_2^2$ ensures that the multiplier estimate $\ys(x;\delta)$ is always defined even when $A(x)$ is rank-deficient. The only computational change is that the $(2,2)$ block of the matrices in \eqref{eq:aug-Yu}--\eqref{eq:aug-mult} and \eqref{eq:aug-proj} is now $-\delta^2 I$.

Besides improving $\hbox{cond}(\Kscr)$, $\delta>0$ has the advantage of making $\Kscr$ symmetric quasi-definite. \cite{Vand:1995} shows that any symmetric permutation of such a matrix possesses an $L D L^T$ factorization with $L$ unit lower triangular and $D$ diagonal indefinite. Result~$2$ of \cite{GillSaunShin:1996} implies that the factorization is stable as long as $\delta$ is sufficiently far from zero. Various authors propose regularized matrices of this type to stabilize optimization methods in the presence of degeneracy. In particular, \cite{Wright:1998} accompanies his discussion with an update scheme for $\delta$ that guarantees fast asymptotic convergence.

We continue to assume that~\eqref{eq:nlp} satisfies~\ref{assump:c3}, but we now replace~\ref{assump:licq-all} by~\ref{assump:licq-min}. 
For a given isolated local minimum $\xstar$ of \eqref{eq:nlp}, $\sigma$ sufficiently large, 
define 
\begin{equation*}
    x(\delta) \in \textstyle\arg\min_x \norm{x - \xstar} \ \mbox{ such that $x$ is a local-min of } \phis(x; \delta)
\end{equation*}
for use as an analytical tool in the upcoming discussion. Although the above argmin may be set-valued, \cref{thm:reg-continuity} shows that for sufficiently small $\delta$, $x(\delta)$ is unique.
 
Note that for $\delta > 0$, we would not expect that $x(\delta) = \xstar$, but we want to ensure that $x(\delta) \rightarrow \xstar$ as $\delta \rightarrow 0$. Note that for $x$ such that $\ys(x)$ is defined,
\begin{align*}
    \ys(x;\delta) &= (A(x)\T A(x) + \delta^2 I)^{-1} A(x)\T A(x) \ys(x)
\\  &= \ys(x) - \delta^2 (A(x)\T A(x) + \delta^2 I)^{-1} \ys(x).
\end{align*}
Therefore for $x$ such that $\phis(x)$ is defined, we can write the
regularized penalty function as a perturbation of the unregularized
one: 
\begin{align}
  \label{eq:perturb-phi}
    \phis(x;\delta) &= f(x) - c(x)\T \ys(x; \delta) \nonumber
\\  &= f(x) - c(x)\T \ys(x) + \delta^2 c(x)\T (A(x)\T A(x) + \delta^2 I)^{-1} \ys(x) \nonumber
\\  &= \phis(x) + \delta^2 P_{\delta}(x),
\end{align}
where $P_{\delta}(x) := c(x)\T (A(x)^T A(x) + \delta^2 I)^{-1} \ys(x)$. By~\ref{assump:c3}, $P_{\delta}$ is bounded and has at least two continuous derivatives in a neighbourhood of $\xstar$. 


\begin{theorem}
\label{thm:reg-continuity}
Suppose~\ref{assump:c3} and~\ref{assump:licq-min} are satisfied, $\xstar$ is a second-order KKT point for \eqref{eq:nlp}, and $\nabla^2 \phis(\xstar) \succ 0$. Then there exists $\deltabar > 0$ such that $x(\delta)$ is a $\Cscr_1$ function for $0 \le \delta < \deltabar$. In particular, $\norm{x(\delta) - \xstar} = O(\delta)$.
\end{theorem}
\begin{proof}
  The theorem follows from the Implicit Function Theorem \citep[Theorem 5.2.4]{Ortega2000} applied to $\nabla \phis (x;\delta) = 0$.
\end{proof}


An option to recover $\xstar$ using $\phis(x;\delta)$ is to minimize a sequence of problems defined by $x_{k+1} = \argmin_x \phis(x;\delta_k)$ with $\delta_k \rightarrow 0$, using $x_k$ to warm-start the next subproblem. However, we show that it is possible to solve a single subproblem by decreasing $\delta$ during the subproblem iterations, while retaining fast local convergence.

To keep results independent of the minimization algorithm being
used, for a family of functions $\mathcal{F}$ we define
$G: \mathcal{F} \times \R^n \rightarrow \R^n$ such that for
$\varphi \in \mathcal{F}$ and an iterate $x$, $G(\varphi,x)$ computes an
update direction. For example, if $\mathcal{F} = \Cscr_2$, we can
represent Newton's method with $G(\varphi,x) = - (\nabla^{2}\varphi(x))^{-1} \nabla \varphi(x)$. 
Define
$\nu(\delta)$ as a function such that for repeated applications,
$\nu^k (\delta) \rightarrow 0$ as $k \rightarrow \infty$ at a chosen
rate; for example, for a quadratic rate, we let
$\nu(\delta) = \delta^2$.

\Cref{alg:regularized} describes how to adaptively update $\delta$ each iteration.

\begin{algorithm}
\caption{Minimization of the regularized penalty function $\phis(x,\delta)$ with $\delta \rightarrow 0$}
\label{alg:regularized}
\begin{algorithmic}[1]
\State Choose $x_1$, $\delta_0 < 1$
\For{$k=1,2,\dots$}
  \State Set
    \begin{equation}
      \label{eq:delta-update}
      \delta_k \gets \max \left\{ \min \left\{ \norm{\nabla \phis(x_k;\delta_{k-1})}, \delta_{k-1} \right\}, \nu (\delta_{k-1}) \right\}
    \end{equation}
  \State $p_k \gets G\left( \phis(\cdot, \delta_k), x_k \right)$
  \State $x_{k+1} \gets x_k + p_k$
\EndFor
\end{algorithmic}
\end{algorithm}

In order to analyze \Cref{alg:regularized}, we formalize the notions of rates of convergence using definitions equivalent to those of \citet[\S 9]{Ortega2000}.

\begin{definition}
  We say that $x_k \rightarrow \xstar$ with order at least
  $\tau > 1$ if there exists $M > 0$ such that, for all sufficiently large \(k\),
  \(\norm{x_{k+1} - \xstar} \le M \norm{x_k - \xstar}^\tau\).
  We say that
  $x_k \rightarrow \xstar$ with $R$-order at least $\tau > 1$ if
  there exists a sequence $\alpha_k$ such that, for all sufficiently large \(k\),
  $$ \norm{x_{k} - \xstar} \le \alpha_k, \qquad \alpha_k \rightarrow 0 \mbox{ with order at least $\tau$.} $$
\end{definition}

We first show that any minimization algorithm achieving a certain
local rate of convergence can be regarded as \emph{inexact Newton}
\citep{DembEiseStei:1982}.
\begin{lemma}
\label{lem:inexact-newton}
Let $\varphi(x)$ be a $\Cscr_2$ function with local minimum $\xstar$ and
$\nabla^2 \varphi(\xstar) \succ 0$. Suppose we minimize $\varphi$ according to
\begin{equation}
  \label{eq:G-iter}
  x_{k+1} = x_k + p_k, \qquad p_k = G(\varphi, x_k),
\end{equation}
such that $x_k \rightarrow \xstar$ with order at least
$\tau \in (1,2]$. Then in some
neighborhood of $\xstar$, the update procedure $G(\varphi, x)$ is equivalent
to the inexact-Newton iteration
\begin{equation}
  \label{eq:newton}
  x_{k+1} \gets x_k + p_k, \qquad \nabla^2 \varphi(x_k) p_k = -\nabla \varphi(x_k) + r_k, \qquad \norm{r_k} = O(\norm{\nabla \varphi(x_k)}^\tau).
\end{equation}
\end{lemma}
\begin{proof}
There exists a neighborhood $N_N(\xstar)$ such that for $x^N_0 \in N_N(\xstar)$, the Newton update $x^N_{k+1} = x^N_k + p^N_k$ with $\nabla^2 \varphi (x^N_k) p^N_k = - \nabla \varphi(x^N_k)$ converges quadratically:
\[
  \norm{\xstar - x_{k+1}^N} \le M_1 \norm{\xstar - x^N_k}^2, \qquad x_k \in N_N(\xstar).
\]
Let $N_G(\xstar)$ be the neighborhood where order $\tau$ convergence is obtained for~\eqref{eq:G-iter} with constant $M_2$. Let $B_{\epsilon}(\xstar) = \{ x \mid \norm{\xstar - x} \le \epsilon \}$. Set $\epsilon < \min\{ M_2^{-1/(\tau-1)}, 1 \}$ such that $B_{\epsilon}(\xstar) \subseteq N_N \cap N_G$, and observe that if $x_0 \in B_{\epsilon}(\xstar)$, then $x_k \in B_{\epsilon}(\xstar)$ for all $k$ because $\norm{\xstar - x_k}$ is monotonically decreasing, because $M_2 \norm{\xstar - x_0}^{\tau-1} < 1$ and so by induction
\[
    \norm{\xstar - x_k} \le M_2 \norm{\xstar - x_{k-1}}^{\tau} = M_2 \norm{\xstar - x_{k-1}}^{\tau-1} \norm{\xstar - x_{k-1}} < \norm{\xstar - x_{k-1}}.
\]
By continuity of \(H(x)\), there exists \(M_3 > 0\) such that $\norm{H(x)} \le M_3$ for all $B_{\epsilon}(\xstar)$. Then for $x_k \in B_{\epsilon}(\xstar)$,
\begin{align*}
    \norm{r_k} = \norm{\nabla^2 \varphi(x_k) p_k + \nabla \varphi(x_k)} &= \norm{\nabla^2 \varphi(x_k) \left( x_{k+1} - x_k - p^N_k\right)}
\\  &\le \norm{\nabla^2 \varphi(x_k)}  \norm{x_{k+1} - \xstar + \xstar - x^N_{k+1}}
\\  &\le M_3 ( \norm{x_{k+1} - \xstar} + \norm{x^N_{k+1} - \xstar})
\\  &\le M_3 (M_1 + M_2) \norm{x_k - \xstar}^\tau,
\end{align*}
Now, because $\varphi \in \Cscr_2$ and $\nabla^2 \varphi(\xstar) \succ 0$,
there exists a constant $M_4$ such that
$\norm{x_k - \xstar} \le M_4 \norm{\nabla \varphi(x_k)}$ for
$x_k \in N_G(\xstar) \cap N_N(\xstar)$. Therefore
$\norm{r_k} \le M_4 M_3 (M_1 + M_2) \norm{\nabla \varphi(x_k)}^\tau$, which is the
inexact-Newton method, convergent with order $\tau$.
\end{proof}
Note that \cref{lem:inexact-newton} can be modified to accommodate any form of superlinear convergence, as long as $\norm{r_k}$ converges at the same rate as $x_k \rightarrow \xstar$. 
\begin{btheorem}
\label{thm:regularized}
Suppose that~\ref{assump:c3} and~\ref{assump:licq-min} are satisfied, $\xstar$ is a second-order KKT point for \eqref{eq:nlp}, $\nabla^2 \phis(\xstar) \succ 0$, and there exists $\deltabar$ and an open set $B(\xstar)$ containing $\xstar$ such that for $\widetilde{x}_0 \in B(\xstar)$ and $0 < \delta \leq \deltabar$, the sequence defined by $\widetilde{x}_{k+1} = \widetilde{x}_k + G(\phis(\cdot; \delta), \widetilde{x}_k)$ converges quadratically to $x(\delta)$:
\[
  \norm{x(\delta) - \widetilde{x}_{k+1}} \le M_{\delta} \norm{x(\delta) - \widetilde{x}_k}^2.
\]
Further suppose that for $\delta \leq \deltabar$, $M_\delta \leq M$ is uniformly bounded.
Then there exists an open set, $B^{\prime}(\xstar)$ that contains $\xstar$, and $0 < \deltaprime < 1$ such that if $x \in \Bprime(\xstar)$, $\delta \leq \deltaprime$ and $x_k \rightarrow \xstar$ for $x_k$ defined by \cref{alg:regularized} (with $\nu(\delta) = \delta^2$), then $x_k \rightarrow \xstar$ R-quadratically.
\end{btheorem}

The proof is in \cref{app:regularized-proof}. Although there are many technical assumptions, the takeaway message is that we need only minimize $\phis(\cdot;\delta_k)$ until $\norm{\nabla \phis} = O(\delta_k)$, because under typical smoothness assumptions we have that
$\norm{x(\delta) - \xstar} = O(\delta)$ for $\delta$ sufficiently small. Decreasing $\delta$ at the same rate as the local convergence rate of the method on a fixed problem should not perturb $\phis(x;\delta)$ too much, therefore allowing for significant progress on the perturbed problem in few steps. 
The assumption that $M_\delta \le M$ uniformly also appears strong, but we believe it is unavoidable---the number of iterations between updates to $\delta$ must be bounded above by a constant for overall convergence to be unimpeded.
Within the basin of convergence and for a fixed $\delta>0$, an optimization method would achieve the same local convergence rate that it would have with $\delta = 0$ fixed.

\cref{thm:regularized} can be generalized to superlinear rates of convergence using a similar proof. As long as $\nu(\cdot)$ drives $\delta \rightarrow 0$ as fast as the underlying algorithm would locally converge for fixed $\delta$, local convergence of the entire regularized algorithm is unchanged. 

\section{Inexact evaluation of the penalty function}
\label{sec:inexact}
We discuss the effects of solving~\eqref{eq:aug-generic}
approximately, and thus evaluating $\phis$ and its derivatives
inexactly. Various optimization solvers can utilize inexact function
values and derivatives while ensuring global convergence and certain
local convergence rates, provided the user can compute relevant
quantities to a prescribed accuracy. For example, \citet[\S
8--9]{ConnGoulToin:2000} describe conditions on the inexactness of
model and gradient evaluations to ensure convergence, and
\citet{HeinVice:2001} and \citet{KourHeinRidzBloe:2014} describe inexact trust-region SQP solvers for
using inexact function values and gradients respectively. We focus on inexactness within trust-region methods for
optimizing $\phis$.

The accuracy and computational cost in the evaluation of $\phis$ and its derivatives depends on the accuracy of the solves of \eqref{eq:aug-generic}.  If the cost to solve \eqref{eq:aug-generic}
depends on solution accuracy (e.g., with iterative linear solvers), it is advantageous to consider optimization solvers that use inexact computations,
especially for large-scale problems.

Let $\Sscr \subseteq \R^n$ be a compact set. In this section, we use
$\wphis(x)$, $\nabla \wphis(x)$, etc.\ to
distinguish the inexact quantities from their exact
counterparts. We also drop the arguments from operators as in
\cref{sec:comp-penalty-funct}. We consider three quantities that are
computed inexactly: $\gs$, $\phis$ and $\nabla \phis$. For given positive error
tolerances $\eta_i$ (which may be relative to their corresponding quantities), we are interested in exploring termination
criteria for solving \eqref{eq:aug-generic} to ensure that the
following conditions hold for all $x \in \Sscr$: \begingroup
\allowdisplaybreaks
\begin{subequations} \label{eq:inexact-evals}
  \begin{align}
      \bigl| \phis - \wphis \bigr| &\leq M \eta_1, \label{eq:inexact-fval}
  \\  \bigl\| \nabla \phis - \nabla \wphis \bigr\| &\leq M \eta_2,  \label{eq:inexact-grad}
  \\  \norm{\gs - \wgs} &\le M \eta_3, \label{eq:inexact-gL}
  \end{align}
\end{subequations}
\endgroup where $M>0$ is some fixed constant (which may or may not be
known). \cite{KourHeinRidzBloe:2014} give a trust-region method using
inexact objective value and gradient information that guarantees
global convergence provided
\eqref{eq:inexact-fval}--\eqref{eq:inexact-grad} hold without
requiring that $M$ be known a priori. We may compare this to the
conditions of \citet[\S 8.4, \S 10.6]{ConnGoulToin:2000}, which
require more stringent conditions on
\eqref{eq:inexact-fval}--\eqref{eq:inexact-grad}. They require
that $\eta_2 = \norm{\nabla \wphis}$ and that $M$ be known and fixed
according to parameters in the trust-region method.

This leads us to the following proposition, which allows us to bound the residuals of \eqref{eq:aug-Yu} and \eqref{eq:aug-mult} to ensure \eqref{eq:inexact-evals}.


\begin{proposition}
\label{prop:inexact-1}
Let $\Sscr$ be a compact set, and suppose that $\sigma_{\min}(A(x)) \geq \lambda > 0$ for all $x \in \Sscr$. Then for $x \in \Sscr$, if
\begin{alignat}{4}
    \norm{r_1} &= \left\| \Kscr \bmat{\wgs \\ \wys} - \bmat{g \\ \sigma c} \right\| &&\leq \min \{1, \norm{c}^{-1} \} \cdot \min \{\eta_1, \eta_3 \}, \label{eq:resid-1}
\end{alignat}
then \eqref{eq:inexact-fval} and \eqref{eq:inexact-gL} hold for some constant $M$.  Also, if
\begin{equation}
    \norm{r_1} \le \eta_2 \text{ and }
    \norm{r_2} = \left\| \Kscr \bmat{\widetilde{v} \\ \widetilde{w}} - \bmat{0 \\ c} \right\|
    \le \min \{1, \eta_2\}, \label{eq:resid-2}
\end{equation}
then \eqref{eq:inexact-grad} holds for some (perhaps different) constant $M$.
\end{proposition}
\begin{proof}
Because $\Sscr$ is compact and $\lambda>0$, there exists $\lambdabar > 0$ such that $\norm{\Kscr}$, $\norm{\Kscr^{-1}} \le \lambdabar$ for all $x \in \Sscr$. Thus,~\eqref{eq:inexact-gL} follows directly from~\eqref{eq:aug-mult} and \eqref{eq:resid-1} with \(M = \lambdabar\).
Similarly,
\begin{align*}
    \bigl| \phis - \wphis \bigr| = \left| c^T \left( \ys - \widetilde{\ys} \right) \right|
  \leq \norm{c} \left\| \ys - \widetilde{\ys} \right\| \leq \lambdabar \eta_1,
\end{align*}
and~\eqref{eq:inexact-fval} holds with \(M = \lambdabar\).
We apply a similar analysis to ensure that \eqref{eq:inexact-grad} holds.
Define the vector
$h \in \R^m$ such that $h_i = \norm{H_i}$. Define $v$, $w$ as the solutions to \eqref{eq:resid-2} for $r_2 = 0$, so that from \eqref{eq:resid-2} we have
\begingroup
\allowdisplaybreaks
\begin{align*}
    \norm{\nabla \phis - \nabla \wphis} &\le \norm{\gs - \wgs} + \norm{\Ys c - \wYs c}
\\  &\le \lambdabar \eta_2 + \norm{(\Hs - \sigma I)v - \Ss\T w - (\wHs - \sigma I)\widetilde{v} + \wSs\T \widetilde{w}}
\\ &\le \lambdabar \eta_2 + \sigma \norm{v - \widetilde{v}} + \norm{\Hs v - \wHs \widetilde{v}} + \norm{\Ss^T w - \wSs\T \widetilde{w}}
\\ &\le \left(\lambdabar + \sigma \lambdabar \right) \eta_2 + \norm{\Hs (v - \widetilde{v}) + (\Hs - \wHs) \widetilde{v}}
\\ &\phantom{\leq \,\,} + \norm{\Ss\T (w - \widetilde{w}) + (\Ss - \wSs)^T \widetilde{w}}
\\ &\le \left(\lambdabar + \sigma \lambdabar \right) \eta_2 + \norm{\Hs} \norm{v - \widetilde{v}}
     + \biggl\| \sum_{i=1}^m \left((\ys)_i - (\wys)_i \right) H_i \biggr\| \norm{\widetilde{v}}
\\ &\phantom{\leq \,\,} + \norm{\Ss} \norm{w - \widetilde{w}} + \biggl\| \sum_{i=1}^m \widetilde{w}_i H_i\biggr\| \norm{\gs - \wgs}
\\ &\le \left(\lambdabar + \sigma \lambdabar + \norm{\Hs} \lambdabar + \lambdabar \norm{\widetilde{v}} \norm{h}
     + \norm{\Ss}\lambdabar + \norm{\widetilde{w}} \|h\| \lambdabar \right) \eta_2.
\end{align*}
\endgroup
Note that $\norm{\Hs}$, $\norm{h}$, $\norm{\Ss}$, $\|\widetilde{w}\|$ and $\|\widetilde{v}\|$ are bounded uniformly in $\Sscr$.
\end{proof}

In the absence of additional information, using \eqref{eq:inexact-evals} with unknown $M$ may be the only way to take advantage of inexact computations, because computing exact constants (such as the norms $\Kscr$ or the various operators above) is not practical. In some cases the bounds \eqref{eq:inexact-evals} are relative, e.g., $\eta_2 = \min \{ \norm{\nabla \wphis}, \Delta \}$ for a certain $\Delta > 0$. It may then be necessary to compute $\norm{\nabla \wphis}$ and refine the solutions of \eqref{eq:aug-mult} and \eqref{eq:aug-Yu} until they satisfy \eqref{eq:resid-1}--\eqref{eq:resid-2}. However, given the expense of applying these operators, it may be more practical to use a nominal relative tolerance, as in the numerical experiments of \cref{sec:numerical-results}.

We include a (trivial) improvement to \cref{prop:inexact-1} that satisfies \eqref{eq:inexact-fval} and \eqref{eq:inexact-gL} with $M=1$, given additional spectral information on $A$. If we solve \eqref{eq:aug-mult} by via
\begin{equation}
    \label{eq:shift-rhs}
    \bmat{I & A \\ A^T & 0} \bmat{ \Delta \gs \\ \ys} = \bmat{0 \\ \sigma c - A\T g}, \qquad \gs = g + \Delta \gs,
\end{equation}
we can use \LNLQ \citep{EstOrbSaun:2019LNLQ}, a Krylov subspace method for
such systems, which ensures that $\norm{\Delta \gs - \Delta \wgs^{(j)}}$
and $\norm{\ys - \wys^{(j)}}$ are monotonic, where $\Delta \wgs^{(j)}$,
$\wys^{(j)}$ are the $j$th \LNLQ iterates.  Given $\lambda>0$ such
that $\sigma_{\min}(A) \geq \lambda$, \LNLQ can compute cheap
upper bounds on $\norm{\Delta \gs - \Delta \wgs^{(j)}}$ and
$\norm{\ys - \wys^{(j)}}$, allowing us to terminate the solve when
$\norm{\Delta \gs - \widetilde{\Delta \gs}} \le \eta_2 \norm{\widetilde{\Delta \gs} + g}
                                              = \eta_2 \norm{\wgs}$ and
$\norm{\ys - \wys} \le \min \{1, \norm{c}^{-1}\} \eta_1$. Typically,
termination criteria for the optimization solver will include a
condition that $\norm{\gs} \le \epsilon_d$ to determine approximate local
minimizers to~\eqref{eq:nlp}. For such cases, we can instead require
that $\norm{\wgs} \le \tfrac{1}{1+\eta_2}\epsd$, because then
\begin{equation*}
    \norm{\gs} \le \norm{\gs - \wgs} + \norm{\wgs} \le (1+\eta_2)\norm{\wgs} \le \epsd.
\end{equation*}
Similarly, we have
\begin{equation*}
    \bigl| \phis - \wphis \bigr| \le \norm{c} \norm{\ys - \wys} \le \eta_1,
\end{equation*}
which now satisfies \eqref{eq:inexact-fval} with $M=1$.

Although finding suitable bounds $\lambda$ on the smallest singular value may be difficult in general, it is trivially available in some cases because of the way $\Kscr$ is preconditioned (for an example, see \cref{sec:numerical-results}). However, a complication is that if \LNLQ is used with a right-preconditioner $\Pscr \approx A\T A$, then $\norm{\ys - \wys}_{\Pscr}$ is monotonic and \LNLQ provides bounds on the preconditioned norm instead of the Euclidean norm. If $\norm{\Pscr^{-1}}$ can be bounded, then the bound $\norm{\ys - \wys} \le \norm{\ys - \wys}_{\Pscr} \norm{\Pscr^{-1}}$ can be used.

\section{Practical considerations}
\label{sec:practical-considerations}

We discuss some matters related to the use of $\phis$ in practice.
In principle, nearly any smooth unconstrained solver can be used to find a local minimum of $\phis$ because it has at least one continuous derivative, and a continuous Hessian approximation if \ref{assump:c2} is satisfied. However, the structure of $\phis$ lends itself more readily to certain optimization methods than to others, especially when the goal of creating a factorization-free solver is kept in mind.

\cite{Fletcher:1973} originally envisioned a Newton-type procedure
\[
  x_{k+1} \gets x_k - \alpha_k B_i^{-1}(x_k) \nabla \phis(x_k), \qquad i=1 \text{or} 2,
\]
where $B_1$, $B_2$ are the Hessian approximations from \eqref{eq:hess-approx} and $\alpha_k > 0$ is a step-size. \citet[Theorem 2]{Fletcher:1973} further proved that superlinear convergence is achieved, or quadratic convergence if the second derivatives of $f$ and $c$ are Lipschitz continuous. However, for large problems it is expensive to compute $B_i$ explicitly and solve the dense system $B_i s_k = -\nabla \phis(x_k)$.

We instead propose using a \cite{Stei:1983} Newton-CG type trust-region solver to minimize $\phis$. First, trust-region methods are preferable to linesearch methods \citep[\S 3--4]{NocedalW:2006} for objectives with expensive evaluations; it is costly to evaluate $\phis$ repeatedly to determine a step-size every iteration as this requires solving a linear system. Further, $\nabla^2 \phis$ is often indefinite and trust-region methods can take advantage of directions of negative curvature. Computing $B_i$ explicitly is not practical, but products are reasonable as they only require solving two additional linear systems with the same matrix, thus motivating the use of a Newton-CG type trust-region solver. In particular, solvers such as TRON \citep{LinMore:1999b} and KNITRO \citep{ByrdNoceWalt:06} are suitable for minimizing $\phis$. KNITRO has the additional advantage of handling explicit linear constraints.

We have not yet addressed choosing the penalty parameter $\sigma$. Although we can provide an a posteriori threshold value for $\sigma^*$, it is difficult to know this threshold ahead of time. \cite{MukaPola:1975} give a scheme for updating $\rho$ with $\phisr$ and $\sigma = 0$; however, they were using a Newton-like scheme that required a solve with $B_1(x)$. Further, $\sigma^*$ ensures only {\em local} convexity, and that a local minimizer of~\eqref{eq:nlp} is a local minimizer of $\phis$---but as with other penalty functions, $\phis$ may be unbounded below in general for any $\sigma$. A heuristic that we employ is to ensure that the primal and dual feasibility, $\norm{c(x)}$ and $\norm{\gLag(x,\ys(x))}$, are within some factor of each other (e.g., 100) to encourage them to decrease at approximately the same rate. If primal feasibility decreases too quickly and small steps are taken, it is indicative of $\sigma$ being too large, and similarly if primal feasibility is too large then $\sigma$ should be increased; this can be done with a multiplicative factor or by estimating $\half \norm{P_A(x) \Hs(x) P_A(x)}$ via the power method (because it is an upper bound on $\sigma^*$ when $x = \xstar$; see \eqref{eq:14}). Although this heuristic is often effective in our experience, in situations where the penalty function begins moving toward negative infinity, we require a different recovery strategy, which is the subject of future work. The work of \cite{MukaPola:1975,ByrdLopeNoce:2012} gives promising directions for developing such strategies.




In practice, regularization (\cref{sec:regularization}) is used only if necessary. For well-behaved problems, using $\delta = 0$ typically requires fewer outer iterations than using $\delta>0$. However, when convergence is slow and/or the Jacobians are ill-conditioned, initializing with $\delta > 0$ is often vital and can improve performance significantly.

\section{Numerical experiments}
\label{sec:numerical-results}

We investigate the performance of Fletcher's penalty function on several PDE-constrained optimization problems and some standard test problems. For each test we use the stopping criterion
\begin{equation}
    \label{eq:experiment-stop}
    \begin{matrix}
       \norm{c(x_k)} \le \epsilon_p
    \\ \norm{\gs(x_k)} \le \epsilon_d
    \end{matrix}
    \qquad \mbox{or} \qquad \norm{\nabla \phis(x_k)} \le \epsilon_d,
\end{equation}
with $\epsilon_p := \epsilon \left( 1 + \norm{x_k}_{\infty} + \norm{c(x_0)}_\infty \right)$ and $\epsilon_d := \epsilon \left( 1 + \norm{y_k}_{\infty} + \norm{\gs(x_0)}_\infty \right)$, where $\epsilon > 0$ is a tolerance, e.g., $\epsilon = 10^{-8}$. We also keep $\sigma$ fixed for each experiment.

Depending on the problem, the augmented systems \eqref{eq:aug-generic}
are solved by either direct or iterative methods. For direct methods,
we use the corrected semi-normal equations \citep{Bjorck:1996}; see
the supplementary materials for implementation details.  For iterative
solves, we use \CRAIG \citep{Ario:2013} with preconditioner $\Pscr$
and two possible termination criteria: for some positive parameter $\eta$,
\begingroup
\allowdisplaybreaks
\begin{subequations} \label{eq:system-solve-accuracy}
  \begin{align}
    \left\| \Kscr \bmat{p^{(k)} \\ q^{(k)}} - \bmat{u \\ v} \right\|_{\widebar\Pscr^{-1}} &\le \eta \left\| \bmat{u \\ v}\right\|_{\widebar\Pscr^{-1}}, \label{eq:terminate-residual}
    \qquad \bar\Pscr := \bmat{I & \\ & \Pscr} 
\\[6pt]  \left\|\bmat{p^* \\ q^*} - \bmat{p^{(k)} \\ q^{(k)}} \right\|_{\widebar\Pscr^{\phantom{-1}}} &\le \eta \left\| \bmat{p^{(k)} \\ q^{(k)}}\right\|_{\widebar\Pscr}, \label{eq:terminate-error}
  \end{align}
\end{subequations}
\endgroup
which are respectively based on the relative residual and the relative error (obtained via \LNLQ). We can use \eqref{eq:terminate-error} when a lower bound on $\sigma_{\min}(A\Pscr^{-1/2})$ is available (e.g., for the PDE-constrained optimization problems).

Because we are using trust-region methods to minimize $\phis$, the objective and gradient of $\phis$ (and therefore of $f$) are evaluated once per iteration. We use KNITRO \citep{ByrdNoceWalt:06} and a Matlab implementation of TRON \citep{LinMore:1999b}. Our implementation of TRON does not require explicit Hessians (only Hessian-vector products) and is unpreconditioned. We use $B_1(x)$ \eqref{eq:hess-approx-1} when efficient products with $S(u,x)$ are available, otherwise we use $B_2(x)$ \eqref{eq:hess-approx-2}. When $\phis$ is evaluated approximately (for large $\eta$), we use the solvers without modification, thus pretending that the function and gradient are evaluated exactly. Note that when the evaluations are noisy, the solvers are no longer guaranteed to converge; the development of suitable solvers is left to future work and discussed in \cref{sec:conclusion}.

\subsection{1D Burgers' problem}
\label{sec:burgers}

\begin{table}[t] \small
  \caption{Results of solving \eqref{eq:burgers} using TRON to minimize $\phis$, with various $\eta$ in \eqref{eq:terminate-error} (left) and \eqref{eq:terminate-residual} (right) to terminate the 
  linear system solves. We record the number of Lagrangian Hessian (\#$Hv$), Jacobian (\#($Av$) and adjoint Jacobian (\#$A\T v$) products.}
  \label{tab:burgers}
  \centering
  \begin{tabular}{@{} c |  c  c  c  c | c  c c c @{}}
    \toprule
    $\eta$ & Iter. & \#$H v$ & \#$Av$ & \#$A\T v$ & Iter. & \#$H v$ & \#$Av$ & \#$A\T v$ \\
    \midrule
    $10^{-2\phantom{0}}$&37 & 19112 & 56797 & 50553    &35 & 7275 & 29453 & 27148\\
    $10^{-4\phantom{0}}$&34 & 6758  & 35559 & 33423    &35 & 7185 & 36757 & 34482\\
    $10^{-6\phantom{0}}$&35 & 7182  & 45893 & 43619    &35 & 7194 & 47999 & 45721\\
    $10^{-8\phantom{0}}$&35 & 7176  & 53296 & 51204    &35 & 7176 & 54025 & 51753\\
    $10^{-10}$          &35 & 7176  & 59802 & 57530    &35 & 7176 & 59310 & 57038\\
    \bottomrule
   \multicolumn{9}{c}{}
 \\[-5pt]\multicolumn{1}{c}{ }
  &\multicolumn{4}{c}{error-based termination}
  &\multicolumn{4}{c}{residual-based termination}
  \end{tabular}
\end{table}

Let $\Omega = (0,1)$ denote the physical domain and $H^1(\Omega)$ denote the Sobolev space of functions in $L^2(\Omega)$, whose weak derivatives are also in $L^2(\Omega)$. 
We solve the following one-dimensional ODE-constrained control problem:
\begin{equation}
\label{eq:burgers}
\begin{aligned}
    \minimize{u \in H^1(\Omega), z \in L^2(\Omega)} &\enspace \half \int_\Omega \left( u(x) - u_d(x) \right)^2 dx + \tfrac{1}{2} \alpha \int_\Omega z(x)^2 dx
\\  \st&\enspace
\begin{array}[t]{rll}
       -\nu u_{xx} + u u_x = & \!\!\!\! z + h & \mbox{in } \Omega,
\\ u(0) = 0, \enspace u(1) = & \!\!\!\! -1,   &
\end{array}
\end{aligned}
\end{equation}
where the constraint is a 1D Burgers' equation over $\Omega = (0,1)$, with $h(x) =  2\left( \nu + x^3 \right)$ and $\nu = 0.08$. The first objective term measures deviation from the data $u_d(x) = -x^2$, while the second term regularizes the control with $\alpha = 10^{-2}$. We discretize~\eqref{eq:burgers} by segmenting $\Omega$ into $n_c=512$ equal-sized cells, and approximate $u$ and $z$ with piecewise linear elements. This results in a nonlinearly constrained optimization problem with $n=2n_c=1024$ variables and $m=n_c-1$ constraints.

We optimize $u,z$ by minimizing $\phis$ with $\sigma = 10^3$, using $B_1(x)$ \eqref{eq:hess-approx-1} as Hessian approximation  and $u_0 = 0$, $z_0 = 0$ as the initial point. We use TRON to optimize $\phis$ and \LNLQ to (approximately) solve~\eqref{eq:aug-generic}. We partition the Jacobian of the discretized constraints into $A(x)\T = \bmat{A_u(x)\T & A_z(x)\T}$, where $A_u(x) \in \R^{n \times n}$ and $A_z(x) \in \R^{m \times m}$ are the Jacobians for $u$ and $z$. We use the preconditioner $\Pscr(x) = A_u(x)\T A_u(x)$, which amounts to performing two linearized Burgers' solves with a
given source. For this preconditioner, $\sigma_{\min}(A\Pscr^{-1/2}) \geq 1$, allowing us to bound the error via \LNLQ and to use both~\eqref{eq:terminate-error} and~\eqref{eq:terminate-residual} to terminate \LNLQ. The maximum number of inner-CG iterations (for solving the trust-region subproblem) is $n$.

We choose $\epsilon = 10^{-8}$ in the stopping
conditions~\eqref{eq:experiment-stop}.  \cref{tab:burgers} records the number of Hessian- and Jacobian-vector products
as we vary the accuracy of the linear system solves via $\eta$ in \eqref{eq:system-solve-accuracy}.

TRON required a moderate number of trust-region iterations.  However,
evaluating $\phis$ and its derivatives can require many Jacobian and
Hessian products, because for every product with the approximate
Hessian we need to solve an augmented linear system.
(Note that there are more products with $A\T$ than $A$ because we shift the right-hand side as in \eqref{eq:shift-rhs} prior to each system solve.)
On the other
hand, the linear systems did not have to be solved to full
precision. As $\eta$ increased from $10^{-10}$ to $10^{-2}$, the
number of Hessian-vector products stayed relatively constant, but the
number of Jacobian-vector products dropped substantially, and the average
number of \LNLQ iterations required per solve dropped from about 9 to
5, except when $\eta = 10^{-2}$ in~\eqref{eq:terminate-error}, the
linear solves were too inaccurate and the number of CG iterations per
trust-region subproblem increased dramatically near the solution
(requiring more linear solves). Using \eqref{eq:terminate-residual}
tended to perform more products with the Lagrangian Hessian and
Jacobian, except when the linear solves were nearly exact, or
extremely inexact.

\subsection{2D Inverse Poisson problem}
\label{sec:inv-poisson}

\begin{table}[t] \small
  \caption{Results of solving \eqref{eq:inverse-poisson} using TRON to minimize $\phis$ with various $\eta$ in \eqref{eq:terminate-error} (left) and \eqref{eq:terminate-residual} (right) to terminate the 
  linear system solves. We record the number of Lagrangian Hessian (\#$Hv$), Jacobian (\#($Av$) and adjoint Jacobian (\#$A\T v$) products.}
  \label{tab:inv-poisson}
  \centering
  \begin{tabular}{@{} c |  c  c  c  c | c c c c @{}}
    \toprule
    $\eta$ & Iter. & \#$H v$ & \#$Av$ & \#$A\T v$ & Iter. & \#$H v$ & \#$Av$ & \#$A\T v$ \\
    \midrule
    $10^{-2\phantom{0}}$&29 & 874 & 1794 & 2608 & 27 & 850 & 1772 & 2562 \\
    $10^{-4\phantom{0}}$&27 & 830 & 1950 & 2728 & 25 & 668 & 1649 & 2265 \\
    $10^{-6\phantom{0}}$&27 & 866 & 2317 & 3129 & 27 & 868 & 2356 & 3168 \\
    $10^{-8\phantom{0}}$&27 & 866 & 2673 & 3485 & 27 & 866 & 2784 & 3596 \\
    $10^{-10}$          &27 & 866 & 3145 & 3957 & 27 & 866 & 3251 & 4063 \\
    \bottomrule
   \multicolumn{9}{c}{}
 \\[-5pt]\multicolumn{1}{c}{ }
  &\multicolumn{4}{c}{error-based termination}
  &\multicolumn{4}{c}{residual-based termination}
  \end{tabular}
\end{table}

Let $\Omega = (-1,1)^2$ denote the physical domain and let $H_0^1(\Omega) \subset H^1(\Omega)$ be the Hilbert space of functions whose value on the boundary $\partial \Omega$ is zero. 
We solve the following 2D PDE-constrained control problem:
\begin{equation}
\label{eq:inverse-poisson}
\begin{aligned}
    \minimize{u \in H_0^1(\Omega), \, z \in L^\infty(\Omega)} &\enspace \half \int_\Omega \left( u - u_d \right)^2 dx + \tfrac{1}{2} \alpha \int_\Omega z^2 dx
\\ \st \quad & \enspace
   \begin{array}[t]{rl}
     - \nabla \cdot (z \nabla u) = h & \mbox{in } \Omega,
   \\                          u = 0 & \mbox{in } \partial \Omega.
   \end{array}
\end{aligned}
\end{equation}
Let $c = (0.2,0.2)$ and define $S_1 = \{ x \mid \norm{x - c}_2 \le 0.3 \}$ and $S_2 = \{x \mid \norm{x - c}_1 \le 0.6 \}$.
The target state $u_d$ is generated as the solution of the PDE with
$z_*(x) = 1 + 0.5 \cdot I_{S_1}(x) + 0.5\cdot I_{S_2}(x)$, where for any set $C$,
$I_C(x) = 1$ if $x \in C$ and 0 otherwise. 

The force term here is $h(x_1, x_2) = - \sin(\omega x_1) \sin(\omega x_2)$, with $\omega = \pi - \tfrac{1}{8}$. The control variable $z$ represents the diffusion coefficients for the Poisson problem that we are trying to recover based on the observed state $u_d$. We set $\alpha=10^{-4}$ as regularization parameter. We discretize \eqref{eq:inverse-poisson} using $P_1$ finite elements on a uniform mesh of 1089 triangular elements and employ an identical discretization for the optimization variables $z \in L^\infty(\Omega)$, obtaining a problem with $n_z = 1089$ controls and $n_u = 961$ states, so that $n=n_u+n_z$. There are $m = n_u$ constraints, as we must solve the PDE on every interior grid point. The initial point is $u_0 = \ind$, $z_0 = \ind$.
Typically $z\ge0$ is explicitly imposed, but we only consider equality constraints in the present paper (inequality constraints are treated in \citep{EstrFrieOrbaSaun:2019b}). For this discretization, the iterates $z_k$ remained positive throughout the minimization.

We use $\sigma=10^{-2}$ as penalty parameter, and $B_2(x)$ as Hessian
approximation. We again use \LNLQ for the linear solves, with the
same preconditioning strategy as in \cref{sec:burgers}. The results
are given in \cref{tab:inv-poisson}. We see a similar trend to that of
\cref{tab:burgers}, as larger $\eta$ allows TRON to converge
within nearly the same number of outer iterations and Lagrangian
Hessian-vector products (even when $\eta=10^{-2}$), while significantly
decreasing the number of Jacobian-vector products. We see again that using
\eqref{eq:terminate-error} to terminate \LNLQ tends to need less work
than with \eqref{eq:terminate-residual}. The exception is
using~\eqref{eq:terminate-residual} with~$\eta=10^{-4}$.
The solver terminates two iterations sooner, resulting in a
sharp drop in Jacobian-vector products but little change in solution quality.
Note that if $\epsilon=10^{-9}$ were used for
the experiment, the runs would appear more similar to one another.

\subsection{2D Poisson-Boltzmann problem}
\label{sec:poisson-boltzmann}

We now solve a control problem where the constraint is a 2D Poisson-Boltzmann equation:
\begin{equation}
\label{eq:poisson-boltzmann}
\begin{aligned}
    \minimize{u \in H_0^1(\Omega), z \in L^2(\Omega)} \enspace & \half \int_\Omega \left( u - u_d \right)^2 dx + \tfrac{1}{2} \alpha \int_\Omega z^2 dx
\\  \st \quad &
    \begin{array}[t]{rll}
         - \Delta u + \sinh(u) = & \hspace{-.5em} h + z & \mbox{in } \Omega,
      \\ u = & \hspace{-.5em} 0 & \mbox{in } \partial \Omega.
    \end{array}
\end{aligned}
\end{equation}
We use the same notation and $\Omega$ as in \cref{sec:inv-poisson}, with forcing term $h(x_1, x_2) = - \sin(\omega x_1) \sin(\omega x_2)$, $\omega = \pi - \tfrac{1}{8}$, and target state 
\[
    u_d(x) = \begin{cases} 10 & \mbox{if } x\in [0.25,0.75]^2 \\
                           5  & \mbox{otherwise.}\end{cases}
\]
We discretize \eqref{eq:poisson-boltzmann} using $P_1$ finite elements on a uniform mesh with 10201 
triangular elements, resulting in a problem with $n=20002$ 
variables and $m=9801$ 
constraints. We use $u_0 = \ind$, $z_0 = \ind$ as the initial point.

\begin{table}[t]\small
  \caption{Results of solving \eqref{eq:poisson-boltzmann} using TRON to minimize $\phis$ with various $\eta$ in \eqref{eq:terminate-error} (left) and \eqref{eq:terminate-residual} (right) to terminate the linear system solves. We record the number of Lagrangian Hessian (\#$Hv$), Jacobian (\#($Av$) and adjoint Jacobian (\#$A\T v$) products.}
  \label{tab:poisson-boltzmann}
  \centering
  \begin{tabular}{@{} c |  c  c  c  c | c c c c @{}}
    \toprule
    $\eta$ & Iter. & \#$H v$ & \#$Av$ & \#$A\T v$ & Iter. & \#$H v$ & \#$Av$ & \#$A\T v$ \\
    \midrule
    $10^{-2\phantom{0}}$&57 & 1150 & 2300 & 3392 & 58 & 1166 & 2427 & 3534 \\
    $10^{-4\phantom{0}}$&58 & 1166 & 2479 & 3586 & 57 & 1150 & 2728 & 3820 \\
    $10^{-6\phantom{0}}$&57 & 1150 & 2812 & 3904 & 57 & 1150 & 3210 & 4302 \\
    $10^{-8\phantom{0}}$&57 & 1150 & 3357 & 4449 & 57 & 1150 & 3721 & 4813 \\
    $10^{-10}$          &57 & 1150 & 3811 & 4903 & 57 & 1150 & 4265 & 5357 \\
    \bottomrule
   \multicolumn{9}{c}{}
 \\[-5pt]\multicolumn{1}{c}{ }
  &\multicolumn{4}{c}{error-based termination}
  &\multicolumn{4}{c}{residual-based termination}
  \end{tabular}
\end{table}

We perform the experiment described in \cref{sec:inv-poisson} using $\sigma=10^{-1}$, and record the results in \cref{tab:poisson-boltzmann}. The results are similar to  \cref{tab:inv-poisson}, where the number of Jacobian products decrease with $\eta$, while the number of outer iterations and Lagrangian-Hessian products stay fairly constant. We see that with stopping criterion \eqref{eq:terminate-residual}, the \LNLQ iterations increase somewhat compared to \eqref{eq:terminate-error}, as it is a tighter criterion.

\subsection{Regularization}

We next solve problems where $A(x)$ is rank-deficient for some iterates, requiring that $\phis$ be regularized (\cref{sec:regularization}). We use the corrected semi-normal equations to solve the linear systems (to machine precision), with $B_2(x)$ as the Hessian approximation.

For problem \texttt{hs061} ($n=3$ variables, $m=3$ constraints) from the CUTEst test set \citep{GoulOrbaToin:2015} 
we use $x_0=0$, $\sigma = 10^2$, $\delta_0 = 10^{-1}$. For problem \texttt{mss1} ($n=90$, $m=73$) we use $x_0=0$, $\sigma=10^3$, $\delta_0=10^{-2}$. In both cases we decrease $\delta$ according to $\nu(\delta) = \delta^2$ to retain local quadratic convergence. For both problems, $A(x_0)$ is rank-deficient and $\phis$ is undefined, so the trust-region solvers terminate immediately. We therefore regularize $\phis$ and record the iteration at which $\delta$ changed. For \texttt{mss1}, we set $\delta_{\min} = 10^{-7}$ to avoid ill-conditioning.  The results are in \cref{tab:hs061}.

\begin{table}[t] \small
    \caption{Convergence results for \texttt{hs061} (left) and \texttt{mss1} (right) when TRON and KNITRO minimize $\phis(\cdot;\delta)$. The first rows show the iteration at which $\delta$ is updated, and the last two rows record the final primal and dual infeasibilities.}
    \label{tab:hs061}
    \centering
    \begin{tabular}{@{}cc@{}}
    \begin{tabular}{@{}c | c c@{}}
    \toprule
    $\delta$ & TRON & KNITRO \\
    \midrule
    $10^{-1}$   & 22 & 12 \\
    $10^{-2}$   & 23 & 13 \\
    $10^{-4}$   & 24 & 14 \\
    \midrule
    $\norm{c(\xbar)}$   & $10^{-10}$ & $10^{-9}$ \\
    $\norm{\gs(\xbar)}$ & $10^{-7}$  & $10^{-8}$ \\
    \bottomrule
    \end{tabular}
    &
    \begin{tabular}{@{}c | c c@{}}
    \toprule
    $\delta$ & TRON & KNITRO \\
    \midrule
    $10^{-2}$   & 46 & 33 \\
    $10^{-4}$   & 52 & 36 \\
    $10^{-7}$   & 53 & 37 \\
    \midrule
    $\norm{c(\xbar)}$   & $10^{-12}$ & $10^{-14}$ \\
    $\norm{\gs(\xbar)}$ & $10^{-8}$  & $10^{-9}$ \\
    \bottomrule
    \end{tabular}
    \end{tabular}
\end{table}

 The regularized problems converge with few iterations between $\delta$ updates, showing evidence of quadratic convergence. Note that a large $\delta$ can perturb $\phis(\cdot;\delta)$ substantially, so that $\delta_0$ may need to be chosen judiciously. We use $\delta_0=10^{-2}$ because neither TRON nor KNITRO would converge for \texttt{mss1} when $\delta_0=10^{-1}$.

\section{Discussion and concluding remarks}
\label{sec:conclusion}

The smooth exact penalty approach is promising for nonlinearly constrained optimization particularly when the augmented linear systems \eqref{eq:aug-generic} can be solved efficiently. However, several potential drawbacks remain as avenues for future work.

One property of $\phis$ observed from our experiments is that it is highly nonlinear and nonconvex. Even though superlinear or quadratic local convergence can be achieved, the high nonlinearity potentially results in many globalization iterations and small step-sizes. 
Further, $\phis$ is usually nonconvex, even if~\eqref{eq:nlp} is convex.

Another aspect not yet discussed is preconditioning the trust-region subproblem.
This is particularly non-trivial because the (approximate) Hessian is available only as an operator.
Traditional approaches based on incomplete factorizations \citep{LinMore:1999} are not applicable. One possibility is to use a preconditioner for the Lagrangian Hessian $\Hs$ as a preconditioner for the approximate penalty Hessian $B_i$ \eqref{eq:hess-approx-1}--\eqref{eq:hess-approx-2}. This may be effective when $m \ll n$  because $\Hs$ and $B_i$ would differ only by a low-rank update; otherwise $\Hs$ can be a poor approximation to $B_i$. Preconditioning is vital for trust-region solvers, and is a direction for future work.

Products with $B_i$ \eqref{eq:hess-approx-1}--\eqref{eq:hess-approx-2} are generally more expensive than typical Hessian-vector products, as they require solving a linear system. Products with a quasi-Newton approximation would be significantly faster. Also, exact curvature away from the solution may be less important than near the solution for choosing good directions; therefore a hybrid scheme that begins with quasi-Newton and switches to $B_i$ may be effective. Another strategy, similar to  \cite{MoraNoce:2000}, is to use quasi-Newton approximations to precondition the trust-region subproblems involving $B_i$: the approximation for iteration $k$ can be updated with every $B_i(x_{k-1})$ product, or with every update step $x_k - x_{k-1}$.

Further improvements that would make our approach applicable to a wider range of problems include: developing suitable solvers based on the work of \citet{HeinVice:2001} and \citet{KourHeinRidzBloe:2014} as these properly handle noisy function and gradient evaluations; a robust update for the penalty parameter $\sigma$; termination criteria on the augmented systems in order to integrate $\phis$ fully into a solver in the style of \cite{KourHeinRidzBloe:2014}; and relaxing \ref{assump:licq-min} to weaker constraint qualifications. Even if $\sigma \ge \sigma^*$ it is possible for $\phis$ to be unbounded, because $\sigma$ only guarantees local convexity. Diverging iterates must be detected sufficiently early because by the time unboundedness is detected, it may be too late to update $\sigma$ and we would need to backtrack several iterates. To relax \ref{assump:licq-min}, it may be possible to combine our regularization \eqref{eq:pen-reg} with a dual proximal-point approach.

The next obvious extension is to handle inequality constraints---the subject of \cite{EstrFrieOrbaSaun:2019b}. \cite{Fletcher:1973} proposed a non-smooth extension to $\phis$ for this purpose, but it is less readily applicable to most solvers.

Our Matlab implementation can be found at \url{https://github.com/optimizers/FletcherPenalty}. To highlight the flexibility of Fletcher's approach, we implemented several options for applying various solvers to the penalty function and for solving the augmented systems, and other options discussed along the way.

\appendix
\section{Technical details}
\label{sec:appendix}
We provide proofs and technical details that were omitted earlier.

\subsection{Example problem with a spurious local minimum}
\label{sec:spurious-min}

Consider the feasibility problem~\eqref{eq:nlp} with $f(x) = 0$ and $c(x) = x^3 + x - 2$. The only minimizer is $x^* = 1$. The penalty function
\[
    \phis(x) = \sigma \frac{(x^3 + x - 2)^2}{(3x^2+1)^2}
\]
is defined everywhere and has local minimizers at $x_1 = 1$ (the solutation) and $x_2 \approx -1.56$. Because the stationary points are independent of $\sigma$ in this case, $\phis$ always has the spurious local minimizer $x_2$.

\subsection{Proof of \cref{thm:regularized}}
\label{app:regularized-proof}

We repeat the assumptions of \cref{thm:regularized}:
\begin{enumerate}[label=(B\arabic*)]
    \item \eqref{eq:nlp} satisfies \ref{assump:c3} and \ref{assump:licq-min}. \label{ass:c3}
    \item $\xstar$ is a second-order KKT point for \eqref{eq:nlp} satisfying $\nabla^2 \phis(\xstar) \succ 0$. \label{ass:licq}
    \item There exist $\deltabar > 0$ and an open set $B(\xstar)$ containing $\xstar$ such that if $\widetilde{x}_0 \in B(\xstar)$ and $\delta \leq \deltabar$, the sequence $\widetilde{x}_{k+1} = \widetilde{x}_k + G(\phis(\cdot; \delta), \widetilde{x}_k)$ converges quadratically to $x(\delta)$ with constant $M$ independent of $\delta$. \label{ass:quad-conv}
    \item Assume $\delta, \delta_k \ge 0$ throughout to avoid indicating this everywhere.
\end{enumerate}

\begin{lemma}
  \label{lem:consequences}
  Under the assumptions of \cref{thm:regularized}:
\begin{enumerate}
    \item \label{cons:1} $\phis(\cdot; \delta)$ has two continuous derivatives for $\delta > 0$ and $x \in \R^n$ by \ref{ass:c3}.
    \item \label{cons:2} There exists an open set $B_1(\xstar)$ containing $\xstar$ such that $\phis(x)$ is well-defined and has two continuous derivatives for all $x \in B_1(\xstar)$ by \ref{ass:c3}.
    \item \label{cons:3} $\nabla^2 \phis(\xstar) = \nabla^2 \phis(\xstar; 0) \succ 0$ and $\phis(x;\delta)$ is $\Cscr_2$ in both $x$ and $\delta$, so by assumption \ref{ass:licq} there exists an open set $B_2(\xstar)$ containing $\xstar$ and $\widetilde{\delta} > 0$ such that $\nabla^2 \phis(x;\delta) \succ 0$ for $(x,\delta) \in B_2(\xstar) \times [0,\widetilde{\delta}]$.
    \item \label{cons:4} By \cref{thm:reg-continuity}, there exists $\deltahat$ such that for $\delta \leq \deltahat$, $x(\delta)$ is continuous in $\delta$. Therefore there exists an open set $B_3(\xstar)$ such that $x(\delta) \in B_3(\xstar)$ for $\delta \leq \deltahat$.
    \item \label{cons:5} There exists a neighborhood $B_4(\xstar)$ where Newton's method is quadratically convergent (with factor $N$) on $\phis(x)$ by \ref{ass:licq}.
    \item \label{cons:6} Given $\delta_0 \leq \min\{\deltabar,\deltahat\}$, where $\deltabar$ is defined in~\ref{ass:quad-conv}, there exists a neighborhood $B_5(\xstar)$ such that $\norm{\nabla \phis(x; \delta_0)} \le \delta_0$ for all $x \in B_5(\xstar)$.
\end{enumerate}
%
%
%
%
We define
\vspace*{-8pt}
\[\Bprime (\xstar) := B(\xstar) \cap \left( \bigcap_{i=1}^5 B_i(\xstar) \right)
  \qquad \mbox{and} \qquad
  \deltaprime < \min \{\deltabar, \widetilde{\delta}, \deltahat, 1\},
\vspace*{-6pt}
\]
and note that $x_k$ defined by \cref{alg:regularized} satisfies $x_k \in \Bprime (\xstar)$ for all $k$ by \ref{ass:quad-conv}.
Because $\phis(x;\delta)$ is $\Cscr_2$ in $\Bprime(\xstar) \times [0,\deltaprime]$, there exist positive constants \(K_1, \dots, K_5\) such that
\begin{enumerate}
    \setcounter{enumi}{6}
    \item \label{cons:7} $\norm{\nabla \phis(x;\delta)} \le K_1 < 1$,
                         $\norm{\nabla^2 \phis(x;\delta)}, \norm{\nabla^2 \phis(x;\delta)^{-1}} \le K_2$ for $x \in \Bprime(\xstar)$ and $\delta \leq \deltaprime$;
    \item \label{cons:8} $\norm{\nabla P_{\delta}(x)} \le K_3$, $\norm{\nabla^2 P_{\delta}(x)} \le K_4$ for $x \in \Bprime(\xstar)$ and $\delta \leq \deltaprime$;
    \item \label{cons:9} $\norm{x_k - \xstar} \le K_5 \norm{\nabla \phis(x_k)}$.
\end{enumerate}
\end{lemma}

\begin{proof}
  Statements~\ref{cons:1}--\ref{cons:2} follow from~\ref{ass:c3}.
  Statements~\ref{cons:3} and~\ref{cons:5} follow from~\ref{ass:licq}.
  Statement~\ref{cons:4} follows from \cref{thm:reg-continuity}.

  Now consider Statement~\ref{cons:6}.
  For a given $\delta_0$, we have $\nabla \phis(x(\delta_0); \delta_0) = 0$ and so there exists a neighbourhood $\widetilde{B}$ around $x(\delta_0)$ such that $\norm{\nabla \phis(x; \delta_0)} \le \delta_0$ for all $x \in \widetilde{B}$. Further, $x(\delta_0) \in B_3(\xstar)$, so let $B_5(\xstar) = \widetilde{B} \cap B_3(\xstar)$.
\end{proof}

We first give some technical results. 
All assume that $x_k \in \Bprime(\xstar)$ and $\delta_k \leq \deltaprime$.

\begin{lemma}
\label{lem:grad}
Assume $\delta_{k-1} \leq \delta_0 \leq \deltaprime$.
For $\delta_k$ defined according to~\eqref{eq:delta-update},
$$ \norm{\nabla \phis(x_k; \delta_k)} = O(\delta_k). $$
\end{lemma}

\begin{proof}
The result holds for $k=0$ in view of observation~\ref{cons:6} of \cref{lem:consequences}.

Because $x_k \in \Bprime(\xstar)$ and $\delta_k$, $\delta_{k-1} \leq \deltaprime$, observation~\ref{cons:8} of \cref{lem:consequences} gives
\(\norm{\nabla P_{\delta_{k-1}} (x_k)} \le K_3\) and
\(\norm{\nabla P_{\delta_k}     (x_k)} \le K_3\).
Using~\eqref{eq:perturb-phi}, we have
\begin{align}
   \norm{\nabla \phis(x_k; \delta_k)} &= \norm{\nabla \phis(x_k;\delta_{k-1}) - \delta_{k-1}^2 \nabla P_{\delta_{k-1}} (x_k) + \delta_k^2 \nabla P_{\delta_k}(x_k)} \nonumber
\\ &\le \norm{\nabla \phis(x_k;\delta_{k-1})} + \delta_{k-1}^2 \norm{\nabla P_{\delta_{k-1}} (x_k)} +
     \delta_k^2 \norm{\nabla P_{\delta_k}(x_k)} \nonumber
\\ &\le \norm{\nabla \phis(x_k;\delta_{k-1})} + (\delta_{k-1}^2 + \delta_k^2) K_3 \nonumber
\\ &=   \norm{\nabla \phis(x_k;\delta_{k-1})} + O(\delta_k), \label{eq:gradphi-intermediate}
\end{align}
where the last inequality follows from $\delta_k^2 \leq \delta_k$ and~\eqref{eq:delta-update}, which implies that $\delta_k \geq \nu (\delta_{k-1}) = \delta_{k-1}^2$.

We consider two cases.
If $\norm{\nabla \phis(x_k; \delta_{k-1})} \le \delta_{k-1}$,~\eqref{eq:delta-update} implies that
\[
    \delta_k = \max(\norm{\nabla \phis(x_k; \delta_{k-1})}, \, \delta_{k-1}^2)
             \ge \norm{\nabla \phis(x_k; \delta_{k-1})},
\]
and therefore \eqref{eq:gradphi-intermediate} gives
\(
    \norm{\nabla \phis(x_k; \delta_k)} = O(\delta_k).
\)

Otherwise, \eqref{eq:delta-update} yields $\delta_k = \max(\delta_{k-1}, \, \delta_{k-1}^2) = \delta_{k-1}$, and there exists $\ell \le k-1$ such that $\delta_{k} = \delta_{k-1} = \dots = \delta_{\ell} < \delta_{\ell-1}$, or $\ell = 1$. If $\delta_{\ell} < \delta_{\ell-1}$, step 3 of \cref{alg:regularized} implies that $\norm{\nabla \phis(x_{\ell};\delta_{\ell-1})} < \delta_{\ell-1}$, which by the above sequence of inequalities implies that $\norm{\nabla \phis(x_{\ell};\delta_\ell)} = O(\delta_{\ell})$. Then, because $\delta_k = \delta_\ell$,
\[
  \norm{\nabla \phis(x_{\ell}, \delta_{k})} = \norm{\nabla \phis(x_{\ell}, \delta_{\ell})} = O(\delta_{\ell}) = O(\delta_k).
\]
Define the sequence $\{\widetilde{x}_j\}$ with $\widetilde{x}_0 = x_\ell$ and $\widetilde{x}_{j+1} = \widetilde{x}_j + G(\phis(\cdot; \delta_\ell), \widetilde{x}_j)$. By \ref{ass:quad-conv}, $\widetilde{x}_j \rightarrow x(\delta_{\ell})$ quadratically, so after $j$ iterations of this procedure, by Taylor's theorem, for some $\Mtilde$, we have
\begin{align*}
    \norm{\nabla \phis(\widetilde{x}_j;\delta_\ell)}       \le \Mtilde^j \norm{\nabla\phis(\widetilde{x}_0;\delta_\ell)}^{2^j} &\le \Mtilde^j K_1^{2^{j-1}}
    \norm{\nabla\phis(\widetilde{x}_0;\delta_\ell)}^2. 
\end{align*}
Then after $j = O(1)$ iterations of this procedure (depending only on $\Mtilde$ and $K_1$), we have $\Mtilde^j K_1^{2^{j-1}} \le 1$ (because $K_1 < 1$), so that
\begin{align*}
    \norm{\nabla \phis(\widetilde{x}_j;\delta_k)} \le \norm{\nabla \phis(\widetilde{x}_0;\delta_k)}^2 = \norm{\nabla \phis(x_{\ell};\delta_k)}^2 = O(\delta_k^2) < O(\delta_k).
\end{align*}
Therefore, $k - \ell \leq j = O(1)$, and by \ref{ass:quad-conv} and the fact that $\delta_{k-1}=\delta_k$,
\[
    \norm{\nabla \phis(x_k; \delta_{k-1})}
    = O\left( M^{k - \ell} \norm{\nabla \phis(x_\ell; \delta_{k-1})}^{2^{k-\ell}} \right)
    = O\left( M^{k - \ell} \delta_{k}^{2^{k-\ell}} \right)
    = O(\delta_k).
\]
The second equality follows from the fact that $\delta_{k-1} = \delta_k$, and so $\norm{\nabla \phis(x_\ell; \delta_{k-1})} = O(\delta_k)$ by the induction assumption.
\end{proof}

\begin{lemma}
\label{lem:step-size}
   For $p_k$ defined by step 4 of \cref{alg:regularized}, $\norm{p_k} = O(\delta_k)$.
\end{lemma}
\begin{proof}
According to \ref{ass:quad-conv}, we may apply \cref{lem:inexact-newton} with $\tau = 2$ and view step 4 of \cref{alg:regularized} as an inexact-Newton step, i.e., there exists a constant $N_2 > 0$ such that
$$ \nabla^2 \phis(x_k;\delta_k) p_k = -\nabla \phis(x_k;\delta_k) + r_k,
   \qquad \norm{r_k} \le N_2 \norm{\nabla \phis(x_k; \delta_k)}^2.
$$
Therefore by \cref{lem:grad},
\begin{align*}
    \norm{p_k} &= \norm{\nabla^2 \phis(x_k;\delta_k)^{-1} (-\nabla \phis(x_k;\delta_k) + r_k)}
\\  &\le \norm{\nabla^2 \phis(x_k;\delta_k)^{-1}} \left( \norm{\nabla \phis(x_k;\delta_k)} + \norm{r_k}\right)
\\  &\le K_2 \left( \norm{\nabla \phis(x_k;\delta_k)} + N_2 \norm{\nabla \phis(x_k; \delta_k)}^2\right)
\\  &\le K_2 \left( O(\delta_k) + O(\delta_k^2) \right) = O(\delta_k).
\end{align*}
\end{proof}

\begin{lemma}
\label{lem:newton-dir}
Let $p_k$ be defined by step 4 of \cref{alg:regularized} and $q_k$ be the Newton direction for the unregularized penalty function defined by $\nabla^2 \phis(x_k) q_k = -\nabla \phis(x_k)$. Then $\norm{p_k - q_k} \in O(\delta_k^2)$.
\end{lemma}
\begin{proof}
According to \ref{ass:quad-conv}, we may apply \cref{lem:inexact-newton} with $\tau = 2$ and view step 4 of \cref{alg:regularized} as an inexact-Newton step, i.e.,
\begin{subequations}
  \begin{align}
    \nabla^2 \phis(x_k;\delta_k) p_k & = -\nabla \phis(x_k;\delta_k) + r_k,
    \label{eq:inexact-newton-tau=2}
    \\ \norm{r_k} &= O(\norm{\nabla \phis(x_k; \delta_k)}^2). \label{eq:rnorm-tau=2}
  \end{align}
\end{subequations}
We premultiply~\eqref{eq:inexact-newton-tau=2} by $\nabla^2 \phis(x_k)^{-1}$ and use~\eqref{eq:perturb-phi} to obtain
\begin{align*}
    p_k + \delta_k^2 \nabla^2 \phis(x_k)^{-1} \nabla^2 P_{\delta_k}(x_k) p_k &= q_k + \delta_k^2 \nabla^2 \phis(x_k)^{-1} \nabla P_{\delta_k}(x_k) + \nabla^2 \phis(x_k)^{-1} r_k.
\end{align*}
\cref{lem:grad}, \cref{lem:step-size} and the triangle inequality then yield
\begin{align*}
    \norm{p_k - q_k} &= \left\| \delta_k^2 \nabla^2 \phis(x_k)^{-1}
                                \left( \nabla P_{\delta_k}(x_k) - \nabla^2 P_{\delta_k}(x_k) p_k \right)
                              + \nabla^2 \phis(x_k)^{-1} r_k \right\|
\\  &\le \delta_k^2 \norm{\nabla^2 \phis(x_k)^{-1}}
                    \left( \norm{\nabla P_{\delta_k}(x_k)}
                         + \norm{\nabla^2 P_{\delta_k}(x_k) p_k} \right)
                         + \norm{\nabla^2 \phis(x_k)^{-1} r_k}
\\  &\le \delta_k^2 K_2 \left( K_3 + O(\delta_k) \right) + O(\delta_k^2) = O(\delta_k^2).
\end{align*}
\end{proof}


Using the previous technical results, we are in position to establish our main result.
\begin{proof}[Proof of \cref{thm:regularized}]
We show that for $x_0 \in \Bprime(\xstar)$ we achieve R-quadratic convergence, by showing that
$\norm{x_k - \xstar} = O(\delta_k)$
and that $\delta_k \rightarrow 0$ quadratically.
By observation~\ref{cons:9} of \cref{lem:consequences}, \eqref{eq:perturb-phi}, the triangle inequality, \cref{lem:grad}, and observation~\ref{cons:8} of \cref{lem:consequences}, we have
\begin{align*}
    \norm{x_k - \xstar} &\le K_5 \norm{\nabla \phis(x_k)}
\\  &=   K_5  \norm{\nabla \phis(x_k; \delta_k)  - \delta_k^2       \nabla P_{\delta_k}(x_k)}
\\  &\le K_5 (\norm{\nabla \phis(x_k; \delta_k)} + \delta_k^2 \norm{\nabla P_{\delta_k}(x_k)})
\\  &\le K_5 \left( O(\delta_k) + \delta_k^2 K_3 \right) = O(\delta_k).
\end{align*}
Let $q_k$ be the Newton direction defined in \cref{lem:newton-dir}.
There exists a constant $N > 0$ such that
\begin{align*}
   \norm{x_{k+1} - \xstar} &=    \norm{x_k + p_k - \xstar}
\\                         &\le  \norm{x_k + q_k - \xstar} + \norm{p_k - q_k}
\\                         &\le N\norm{x_k - \xstar}^2 + \norm{p_k - q_k} = O(\delta_k^2).
\end{align*}
It remains to show that $\delta_k$ decreases quadratically.
If $\norm{\nabla \phis(x_{k+1}, \delta_k)} \le \delta_k^2$,
\[
    \delta_{k+1} =
    \max \{ \min \{\norm{\nabla \phis(x_{k+1}, \delta_k)}, \delta_k \}, \delta_k^2 \} \le
    \max \{ \norm{\nabla \phis(x_{k+1}, \delta_k)}, \delta_k^2 \} = \delta_k^2.
\]
Assume now that $\norm{\nabla \phis(x_{k+1}, \delta_k)} > \delta_k^2$.
We have from~\eqref{eq:perturb-phi} and observations~\ref{cons:7}--\ref{cons:8} of \cref{lem:consequences} that
\begin{align*}
    \delta_{k+1} &= \max \{ \min \{ \norm{\nabla \phis(x_{k+1}, \delta_k)}, \delta_k \}, \delta_k^2 \}
\\  &\le \norm{\nabla \phis(x_{k+1}, \delta_k)}
\\  &\le \norm{\nabla \phis(x_{k+1})} + \delta_k^2 \norm{\nabla P_{\delta_k} (x_{k+1})}
\\  &\le K_2^{-1} \norm{x_{k+1} - \xstar} + \delta_k^2 K_3 = O(\delta_k^2).
\end{align*}
Thus we have $\norm{x_k - \xstar} = O(\delta_k)$ and $\delta_{k+1} = O(\delta_k^2)$, which means that $x_k \rightarrow \xstar$ R-quadratically.
\end{proof}

\section*{Acknowledgements}
We would like to express our deep gratitude to Drew Kouri for supplying PDE-constrained optimization problems in Matlab, for helpful discussions throughout this project, and for hosting the first author for two weeks at Sandia National Laboratories.
We are also extremely grateful to the reviewers for their unusually careful reading and their many helpful questions and suggestions.

\frenchspacing
\small
\bibliographystyle{abbrvnat}  
\bibliography{fletcher}
\normalsize


\end{document}